\documentclass[11pt]{amsart}
\usepackage{amsfonts}
\usepackage{amsmath}
\usepackage{amsthm}
\usepackage{bm,bbm}
\usepackage[dvips]{graphicx}
\usepackage{caption}
\usepackage{color}

\usepackage{bigints}

\usepackage{mathtools}
\usepackage{commath}

\usepackage[hypertexnames=false,linktocpage=true, bookmarks=true]{hyperref}
\hypersetup{colorlinks=true,linkcolor=blue,anchorcolor=blue,citecolor=blue,filecolor=blue, urlcolor=blue,bookmarksnumbered=true,pdfstartview=FitB}

\numberwithin{equation}{section}

\allowdisplaybreaks

\theoremstyle{plain}
\newtheorem{theorem}{Theorem}[section]
\newtheorem{lemma}[theorem]{Lemma}
\newtheorem{proposition}[theorem]{Proposition}
\newtheorem{corollary}[theorem]{Corollary}
\theoremstyle{definition}

\def\r{\rho}

\def\th{\theta}

\def\oi{[0, 1]^d}

\def\ti{\times}
\def\tk{\tau_k}
\def\k{\kappa}

\def\ff{\infty}

\def\ois{E_0}

\def\d{\dif}
\def\ms{\mathsf}
\def\beqn{\begin{equation}}
\def\beqn*{$$}
\def\eeqn{\end{equation}}

\def\P{\mathbb{P}}

\def\E{\mathbb{E}}
\def\Pn{\mathcal P_n}

\newcommand{\reals}{{\mathbb R}}

\newcommand{\R}{\reals}
\newcommand{\bbn}{{\mathbb N}}

\newcommand{\vep}{\varepsilon}
\newcommand{\bbz}{\protect{\mathbb Z}}

\newcommand{\I}{\mathcal I}

\newcommand{\B}{\mathcal B}

\newcommand{\bell}{{\bm \ell}}

\newcommand{\one}{{\mathbbm 1}}

\newcommand{\remove}[1]{}

\def\enp{\end{proof}}
\def\bel{\begin{lemma}}
\def\bep{\begin{proof}}
\def\enl{\end{lemma}}
\newcommand{\M}{\mathcal M}

\newcommand{\Pnr}{\mathcal P_n|}

\newcommand{\QmM}{K_\ell}
\newcommand{\Leb}{\ms{Leb}}
\newcommand{\kln}{\kappa_{\ell,n}}
\newcommand{\ktln}{\widetilde \kappa_{\ell,n}}
\newcommand{\D}{\mathcal D}

\begin{document}

\bibliographystyle{abbrv}

\renewcommand{\baselinestretch}{1.05}

\title[Large deviations for $k$-nearest neighbor balls]
{Large deviations  for the volume of $k$-nearest neighbor balls}

\author{Christian Hirsch}
\address{Department of Mathematics\\ Aarhus University \\  Ny Munkegade, 118, 8000, Aarhus C,  Denmark.}
\email{hirsch@math.au.dk}

\author{Taegyu Kang}
\address{Department of Statistics\\
Purdue University \\
West Lafayette, 47907, USA}
\email{kang426@purdue.edu}

\author{Takashi Owada}
\address{Department of Statistics\\
Purdue University \\
West Lafayette, 47907, USA}
\email{owada@purdue.edu}

\thanks{CH would like to acknowledge the financial support of the CogniGron research center and the Ubbo Emmius Funds (Univ.~of Groningen). TO's research was partially supported by the AFOSR grant FA9550-22-0238  and the NSF grant DMS-1811428.}

\subjclass[2010]{Primary 60F10. Secondary 60D05, 60G55}
\keywords{Large deviation principle, $\mathcal M_0$-convergence, $k$-nearest neighbor ball, point process, stochastic geometry}

\begin{abstract}
This paper develops the large deviations theory for the point process associated with the Euclidean volume of $k$-nearest neighbor balls centered around the points of a homogeneous Poisson or a binomial point processes in the unit cube. Two different types of large deviation behaviors of such point processes are investigated. Our first result is the Donsker-Varadhan large deviation principle, under the assumption that the centering terms for the volume of $k$-nearest neighbor balls grow to infinity more slowly than those needed for Poisson convergence. Additionally, we also study large deviations based on the notion of $\mathcal M_0$-topology, which takes place when the centering terms tend to infinity sufficiently fast, compared to those for Poisson convergence. As applications of our main theorems, we discuss large deviations for the number of Poisson or binomial points of degree at most $k$ in a geometric graph in the dense regime. 

\end{abstract}

\maketitle

\section{Introduction}

The main theme of this paper is to develop the large deviations theory for the point process associated with the Euclidean volume of $k$-nearest neighbor balls. 
We consider the unit cube $[0,1]^d$ equipped with the toroidal metric 
$$
\ms{dist} (x,y) = \min_{z\in \bbz^d} \| x-y+z \|, 
$$
where $\| \cdot \|$ denotes the Euclidean metric in $\R^d$. 
Then, the unit cube $[0,1]^d$ is considered as a flat torus with a periodic boundary. 
Let $B_r(x) = \big\{ y\in [0,1]^d: \ms{dist} (x,y)\le r \big\}$, $r>0$, be the closed ball in $[0,1]^d$ of radius $r$ centered at $x\in[0,1]^d$. 
Given a point measure $\omega$ in $[0,1]^d$ and a point $x\in [0,1]^d$, and a fixed integer $k\ge1$, define 
\begin{equation} \label{e:def.k.NN.dist}
R_k(x,\omega) := \inf \big\{ r>0: \omega\big( B_r(x)\setminus \{x\} \big) \ge k \big\}
\end{equation}
to be  the \emph{$k$-nearest neighbor distance} of $x$; it gives a radius $r$ for  which $B_r(x)\setminus \{x\}$ contains exactly $k$ points of $\omega$ with one of those points lying on the boundary of  $B_r(x)$. 

Given a homogeneous Poisson point process $\Pn$ on  $[0,1]^d$ with intensity $n$, we are interested in the stochastic behavior of the point process
\begin{equation}  \label{e:def.L.kn}
	L_{k,n} := \begin{cases} 
	\sum_{X\in \Pn}\delta_{(X, n\theta_d R_k(X,\Pn)^d -a_n)} & \text{ if } |\Pn|>k,\\
	\emptyset & \text{ if } |\Pn| \le k, 
	\end{cases}
\end{equation}
where $\delta_{(x,y)}$ is the Dirac measure at $(x,y)\in [0,1]^d\times \R$, and $\theta_d$ is volume of the unit ball in $\R^d$, so that $\theta_d R_k(X,\Pn)^d$ represents the volume of a $k$-nearest neighbor ball centered at $X\in\Pn$. 
Further, $\emptyset$ represents the null measure, i.e., the measure assigning zeros to all Borel measurable sets. 
The process \eqref{e:def.L.kn} dictates the magnitude of the normalized volume of $k$-nearest neighbor balls, as well as the location of their centers. 

There have been a number of  studies on the asymptotics  of the process  \eqref{e:def.L.kn} or its similar  variant, when the centering term $a_n$ is given as 
\begin{equation}  \label{e:Poisson.centering}
a_n = \log n + (k-1) \log \log n + \text{constant}. 
\end{equation}
In this case,  one may observe, asymptotically,  at most \emph{finitely many} $k$-nearest neighbor balls whose volume are approximately  $a_n/n$ up to the scale. As a consequence, the process \eqref{e:def.L.kn} will have a Poissonian structure in the limit \cite{penrose:1997, gyorfi:henze:walk:2019, chenavier:henze:otto:2022, otto:2020, bobrowski:schulte:yogeshwaran:2021}. In particular,   the rate of Poisson convergence has recently been derived in terms of the Kantorovich-Rubinstein distance \cite{bobrowski:schulte:yogeshwaran:2021} and the total variation distance \cite{otto:2020}. 

In addition to these results on Poisson convergence, there have also been many attempts at deriving other limit theorems for the functional of a $k$-nearest neighbor distance in  \eqref{e:def.k.NN.dist}, among them, central limit theorems in \cite{penrose:2000, penrose:2007a} and   laws of large numbers in \cite{penrose:yukich:2003, penrose:2007b, penrose:yukich:2011}. For example, Penrose \cite{penrose:2000} proved a (functional) central limit theorem for the number of inhomogeneous Poisson points $X_i$ with density $f$, such that $f(X_i)R_{k_n}(X_i,\Pn)^d$ does not exceed certain thresholds ($k_n$ is taken to be a function of $n$). Additionally, Penrose and Yukich \cite{penrose:yukich:2011} provided laws of large numbers for the sum of power weighted nearest neighbor distances. As for the large deviation results on the $k$-nearest neighbor distance, Schreiber and Yukich \cite{schreiber:yukich:2005} obtained  the \emph{Donsker-Varadhan large deviation principle} (LDP) (see Section 1.2 in \cite{dembo:zeitouni:1998} for a precise definition) for the functional of the length of  edges in the $k$-nearest neighbor graph in $\R^d$. This was  obtained as an application of a more general LDP for the spatial point process satisfying a weak dependence condition characterized by a radius of stabilization. Moreover, Hirsch et al.~\cite{hirsch:jahnel:tobias:2020} analyzed  lower tail large deviations for  general geometric functionals, including  the power-weighted edge lengths in the $k$-nearest neighbor graph. 

The primary objective of this paper is to provide comprehensive results on the asymptotics of the process \eqref{e:def.L.kn}, from the viewpoints of large deviations. We consider two distinct scenarios with respect to a divergence speed of the centering term $(a_n)$. The first scenario examined in this paper  is that 
\begin{equation}  \label{e:cond1.an.intro}
a_n \to\infty,  \ \ \ a_n - \log n - (k-1) \log \log n \to -\infty, \ \ \text{as } n\to\infty. 
\end{equation}
In this case, $(a_n)$ grows to infinity more slowly than  \eqref{e:Poisson.centering}. Intrinsically, there   appear  \emph{infinitely many} $k$-nearest neighbor balls as $n\to\infty$, whose volume are approximately $a_n/n$ up to the scale. Then, the process \eqref{e:def.L.kn} ``diverges" in the limit, in the sense that  $L_{k,n}(A)\to\infty$ as $n\to\infty$, for all (nice) measurable sets $A$. Thus, in order to  dictate its  large deviation behavior, one has to scale the process  \eqref{e:def.L.kn} by some growing sequence $u_n\to\infty$. 
More concretely, 
we aim to establish the Donsker-Varadhan LDP for the properly  scaled process $(L_{k,n}/u_n)_{n\ge1}$. 

In the second scenario of this paper, we consider the centering term $(a_n)$ satisfying 
\begin{equation}  \label{e:cond2.an.intro}
a_n - \log n - (k-1) \log \log n \to \infty, \ \ a_n=o(n), \ \ \text{as } n\to\infty. 
\end{equation}
Then, $(a_n)$ tends to infinity more rapidly than  \eqref{e:Poisson.centering}, so that the $k$-nearest neighbor balls centered around $\Pn$, whose  volume are approximately $a_n/n$ up to the scale, are even less likely to occur. In other words, the occurrence of such $k$-nearest neighbor balls is a ``rare event", in the sense of $\P\big(L_{k,n}(A)\ge1\big)\to0$ as $n\to\infty$, for all measurable sets $A$. In this setup, we present the other type of large deviation results, by detecting a sequence $v_n\to\infty$, such that 
\begin{equation}  \label{e:M0.seq.intro}
\big( v_n \P(L_{k,n}\in \cdot), \, n\ge1 \big)
\end{equation}
converges to a (non-trivial) limit measure. The topology underlying the convergence of \eqref{e:M0.seq.intro} is \emph{$\M_0$-topology}. The notion of $\M_0$-topology was first developed by \cite{hult:lindskog:2006a}. Since then, it has been used mainly for the study of regular variation of stochastic processes \cite{hult:samorodnitsky:2010, lindskog:resnick:roy:2014, fasen:roy:2016, segers:zhao:meinguet:2017, owada:2022b}. 

For the required LDP in Theorem \ref{t:LDP.L.kn} below, many of the techniques in our previous work \cite{hirsch:owada:2022} will be exploited. We first partition the unit cube $[0,1]^d$ into smaller cubes of equal volume, and define a collection of i.i.d.~point processes restricted to each of the small cubes. Next, using one of the main results in \cite{bobrowski:schulte:yogeshwaran:2021}, Proposition \ref{p:KR.conv1} proves that the law of these point processes restricted to small cubes converges to the law of certain Poisson point processes, in terms of the Kantorovich-Rubinstein distance. Subsequently, Proposition \ref{p:eta.zeta.negligible} justifies that this approximation is still feasible even for the version of the empirical measures. The main machinery here is the notion of maximal coupling, provided in \cite[Lemma 4.32]{kallenberg}. Other approximation arguments necessary for our proof will be completed in a series of results in Propositions \ref{p:exp.neg.eta.etap}--\ref{p:exp.neg.etap.Lp}. As a final note, we want to emphasize that the homogeneity assumption of $\Pn$ is crucial throughout our  proof. We anticipate that the LDP still holds even when $\Pn$ in \eqref{e:def.L.kn} is replaced by an inhomogeneous Poisson point process. It seems, however, that unlike the previous studies \cite{penrose:1997, chenavier:henze:otto:2022, otto:2020, bobrowski:schulte:yogeshwaran:2021}, this extension should require much more involved machinery; this will be left as a topic of future research. 

The rest of the paper is outlined as follows. In Section \ref{sec:LDP}, under the assumption \eqref{e:cond1.an.intro}, we give a precise setup for the point process \eqref{e:def.L.kn} and formalize the desired LDP. Section \ref{sec:M0} assumes condition \eqref{e:cond2.an.intro} and establishes the $\M_0$-convergence for the sequence   \eqref{e:M0.seq.intro}. In both settings, we also consider the case that the point processes  are generated by a binomial point process.   Unfortunately, in the context of LDPs, there are no unified results on the De-Poissonization scheme, such as \cite[Section 2.5]{penrose:2003}, which may allow us to extend the LDPs with a Poisson input to those with a binomial input.  Alternatively, we have proved directly the desired exponential equivalence (in terms of the total variation distance) between  processes with a Poisson input and those with a binomial input (see Corollary \ref{c:LDP.LknB}). Corollary \ref{c:M0.LknB} applies a similar kind of De-Poissonization machinery to the  $\M_0$-convergence in Theorem \ref{t:M0.L.kn}. 
Finally, as an application of Theorem \ref{t:LDP.L.kn} and Corollary \ref{c:LDP.LknB}, we deduce the LDP for the number of Poisson (or binomial) points of degree at most $k$ in a geometric graph in the dense regime. Additional applications in a similar vein from the perspective of $\M_0$-convergence, can be found in Corollary \ref{c:M0.application}. 
\medskip
\section{Large deviation principle}  \label{sec:LDP}

Given a point measure $\omega$ in $[0,1]^d$ and $x\in [0,1]^d$, define the $k$-nearest neighbor distance $R_k(x,\omega)$ as in \eqref{e:def.k.NN.dist}. 
Let $(a_n)_{n\ge1}$ be a sequence tending to infinity, such that 
\begin{equation}  \label{e:speed.an}
a_n - \log n -(k-1)\log \log n \to -\infty,  \ \ \ n\to\infty. 
\end{equation}
Given a homogeneous Poisson point process $\Pn$ on  $[0,1]^d$ with intensity $n$,  define the point process  \eqref{e:def.L.kn}
on the space  $M_p\big(\ois \big)$ of point measures on $\ois := [0, 1]^d\ti [s_0,\ff)$, where $s_0\in \R$ is a fixed real number. 
Our aim is to explore the large deviation behavior of $(L_{k,n})_{n\ge1}$. 
More specifically,  we define the scaling constants
\begin{equation}  \label{e:def.bn}
b_n = na_n^{k-1}e^{-a_n}, \ \ \ n\ge1, 
\end{equation}
and establish the LDP for $(L_{k,n}/b_n)_{n\ge1}$ in the space $M_+(\ois)$ of Radon measures on $\ois$. The space $M_+(\ois)$ is equipped with the weak topology. 
Under the assumption \eqref{e:speed.an}, it is elementary  to show  that $b_n\to\infty$ as $n\to\infty$. 

To state our LDP more precisely, we introduce the measure 
$$
\tk(\dif u) := \frac{ e^{-u} }{(k-1)!}\, \one \{ u\ge s_0 \}\dif u. 
$$
Here,  we provide two equivalent representations of the rate function. The first is based on the relative entropy. More precisely, writing $\Leb \otimes \tau_k$ for the product measure of the Lebesgue measure on $\oi$ and $\tk$, define the relative entropy of  $\rho \in M_+(\ois)$ with respect to $\Leb \otimes \tau_k$:
\begin{equation}  \label{e:def.relative.entropy}
H_k(\r \,| \, \Leb \otimes \tau_k) := \int_{\ois} \log\Big\{\frac{\d \r (x,u)}{\d \, (\Leb \otimes \tau_k)}\Big\} \r(\d x, \d u) - \r\big(\ois\big) + (\Leb \otimes \tau_k)(\ois), \ \ \text{if } \rho \ll \ms{Leb}\otimes \tau_k, 
\end{equation}
and $H_k(\r \,| \, \Leb \otimes \tau_k)=\infty$ otherwise.  The second representation of the rate function is given in terms of the Legendre transform. More concretely, 
$$
\Lambda_k^*(\rho):=\sup_{f\in C_b(\ois)} \Big\{  \int_{\ois} f(x, u)\rho(\dif x, \d u) - \int_{\ois} \big(e^{f(x, u)}-1 \big) \dif x\, \tk(\d u) \Big\},  \ \ \rho\in M_+\big(\ois\big), 
$$
where $C_b(\ois)$ is the space of continuous and bounded real-valued functions on $\ois$. 

The proofs of the results in this section are all deferred to Section \ref{sec:proof.LDP}. 
\begin{theorem}  \label{t:LDP.L.kn}
The sequence $(L_{k,n}/b_n)_{n\ge1}$ satisfies an LDP on $M_+\big( \ois \big)$ in the weak topology,  with rate $b_n$ and  rate function $H_k(\cdot\, | \, \Leb \otimes \tau_k)= \Lambda_k^*$.
\end{theorem}

The corollary below extends Theorem \ref{t:LDP.L.kn} to the case that the process is generated by a binomial point process $\B_n:=\{ X_1,\dots,X_n \}$. Precisely, we define, in the space $M_p(E_0)$, 
$$
L_{k,n}^\ms{B} := \sum_{X\in \B_n} \delta_{(X, n\theta_d R_k(X, \B_n)^d-a_n)}, \ \ n > k,
$$
and $L_{k,n}^\ms{B} \equiv \emptyset$ for $n \le k$.

\begin{corollary}  \label{c:LDP.LknB}
The sequence $(L_{k,n}^\ms{B}/b_n)_{n\ge1}$ satisfies an LDP on $M_+(E_0)$   in the weak topology, with rate $b_n$ and  rate function $H_k(\cdot\, | \, \Leb \otimes \tau_k)= \Lambda_k^*$.
\end{corollary}
\medskip

As an application of Theorem \ref{t:LDP.L.kn}, we deduce the LDP for $(T_{k,n}/b_n)_{n\ge1}$, where 
\begin{equation}  \label{e:def.T.kn}
T_{k,n} := \sum_{X\in \Pn}\one \Big\{ \Pn \big( B_{r_n(s_0)}(X) \big)\le k \Big\}, \ \ n\ge1. 
\end{equation}
The statistics \eqref{e:def.T.kn} represents the number of Poisson points of degree at most $k$ in a  geometric graph of  vertex set $\Pn$ and edges between $X_i$ and $X_j$ satisfying $\|X_i-X_j\| \le r_n(s_0)$, where
\begin{equation}  \label{e:def.rn.s0}
r_n(s_0) := \Big( \frac{a_n+s_0}{n\theta_d} \Big)^{1/d}. 
\end{equation}
The threshold radius in  \eqref{e:def.rn.s0}  ensures that the geometric graph under consideration  is of the \emph{dense regime}, such that  $nr_n(s_0)^d =(a_n+s_0)/\theta_d\to \infty$ as $n\to\infty$. Replacing $\Pn$ in \eqref{e:def.T.kn} with its binomial counterpart $\B_n$, we also derive the LDP for $(T_{k,n}^\ms{B}/b_n)_{n\ge1}$, where 
$$
T_{k,n}^\ms{B} := \sum_{X\in \B_n}\one \Big\{ \B_n \big( B_{r_n(s_0)}(X) \big)\le k \Big\}, \ \ n\ge1. 
$$

\begin{corollary}  \label{c:LDP.application}
The sequence $(T_{k,n}/b_n)_{n\ge1}$ satisfies an LDP with rate $b_n$ and rate function 
$$
I_k(x)=\begin{cases}
x\log \big( x/\alpha_k \big)-x+\alpha_k  &\text{ if } x\ge0, \\
\infty &\text{if } x < 0, 
\end{cases}
$$
where $\alpha_k=e^{-s_0}/(k-1)!$. Furthermore, $(T_{k,n}^\ms{B}/b_n)_{n\ge1}$ satisfies the same LDP as $(T_{k,n}/b_n)_{n\ge1}$. 
\end{corollary}

In the above, $I_k(x)$ coincides with a rate function in the LDP for $\big( n^{-1}\sum_{i=1}^n Y_i \big)_{n\ge1}$ where the $Y_i$ are i.i.d.~Poisson with mean $\alpha_k$. 
\medskip

\section{Large deviation under $\M_0$-topology}  \label{sec:M0}

In this section, we explore the large deviation behavior of the process \eqref{e:def.L.kn}, in the case that  $a_n$ tends to infinity more rapidly than in  the last section. Namely, we assume that $(a_n)_{n\ge1}$ satisfies
$$
a_n -\log n -(k-1)\log \log n \to \infty, \ \ \ n\to\infty. 
$$
We again introduce the sequence $b_n=na_n^{k-1}e^{-a_n}$ as in \eqref{e:def.bn}. However, unlike in the last section, 
$b_n \to0$ as $n\to\infty$, because 
\begin{align*}
b_n &= e^{-(a_n -\log n -(k-1)\log \log n )} \Big(  \frac{a_n}{\log n} \Big)^{k-1} \\
&\le C  e^{-(a_n -\log n -(k-1)\log \log n )}  \big( a_n -\log n -(k-1)\log \log n \big)^{k-1} \to 0, 
\end{align*}
for some $C>0$. 

Our  objective  is to investigate large deviations for the sequence  $(\P\circ L_{k,n}^{-1})_{n\ge1}$ of probability distributions of $(L_{k,n})_{n\ge1}$ on the space $M_p(E)$, where $E:= [0,1]^d\times (-\infty, \infty]$. A main challenge is that the space $M_p(E)$ is not locally compact, and therefore, the vague topology would no longer be applicable for the convergence of such probability distributions. To overcome this difficulty, we  adopt the notion of \emph{$\M_0$-topology}. The main feature of $\M_0$-topology is that the corresponding test functions are continuous and bounded real-valued functions on $M_p(E)$ that vanish in the neighborhood of the origin. For the space $M_p(E)$, one can take the null measure $\emptyset$ as its origin. Let  $B_{\emptyset, r}$ denote an open ball of radius $r>0$ centered at $\emptyset$ in the vague metric. 
Denote by $\M_0 = \M_0\big( M_p(E) \big)$ the space of Borel measures on $M_p(E)$, the restriction of which to $M_p(E)\setminus B_{\emptyset, r}$ is finite for all $r>0$. Moreover, define $\mathcal C_0 = \mathcal C_0\big(M_p(E) \big)$ to be the space of continuous and bounded real-valued functions on $M_p(E)$ that vanish in the neighborhood of $\emptyset$. Given $\xi_n, \xi \in \M_0$, we say that $\xi_n$ converges to $\xi$ in the $\M_0$-topology, denoted as $\xi_n\to \xi$ in $\M_0$, if it holds that $\int_{M_p(E)}g(\eta)\xi_n(\dif \eta) \to \int_{M_p(E)}g(\eta)\xi(\dif \eta)$ for all $g\in \mathcal C_0$. For more information on $\M_0$-topology we refer to \cite{hult:lindskog:2006a}. 

Before stating the main theorem,  we will  impose an additional condition that   $a_n=o(n)$ as $n\to\infty$. To see  the necessity of   this  assumption, suppose, to the contrary, that $a_n/n\to\infty$ as $n\to\infty$. Then,  
it trivially holds that  $L_{k,n}\big( [0,1]^d \times (-M,\infty] \big)=0$ a.s.~for large enough $n$ and  any $M>0$. By putting the assumption $a_n=o(n)$ as above, one can  exclude such triviality. 

The proofs of the results below are all  given in Section \ref{sec:proof.M0}. 

\begin{theorem}  \label{t:M0.L.kn}
In the above setting with $a_n = o(n)$, as $n\to\infty$, 
\begin{equation}  \label{e:M0.convergence}
b_n^{-1}\P (L_{k,n}\in \cdot ) \to \xi_k,   \ \ \text{in } \M_0, 
\end{equation}
where 
$$
\xi_k(\cdot) := \frac{1}{(k-1)!} \int_E \one \{ \delta_{(x,u)} \in \cdot \} \, e^{-u} \dif x\dif u. 
$$
\end{theorem}

Additionally, we  present an analogous result for the process $(L_{k,n}^\ms{B})_{n\ge1}$ as well. For a precise statement, however, we need to put a more stringent  condition on $(a_n)$ for the purpose of proving   \eqref{e:Fnvep.M0} in Section \ref{sec:proof.M0}. 

\begin{corollary}  \label{c:M0.LknB}
In the above setting with   $a_n=o(n^{1/3})$, as $n\to\infty$, 
$$
b_n^{-1} \P(L_{k,n}^\ms{B}\in \cdot) \to \xi_k, \ \ \text{ in } M_0. 
$$
\end{corollary}

Finally, certain asymptotic results on $(T_{k,n})_{n\ge1}$ in \eqref{e:def.T.kn}, as well as those on $(T_{k,n}^\ms{B})_{n\ge1}$, are presented as a corollary of the above results. This corollary gives the exact rate (up to the scale) of a probability that the number of Poisson (or binomial) points of degree at most $k$ becomes non-zero. 
\begin{corollary}  \label{c:M0.application}
$(i)$ If $a_n=o(n)$, then 
$$
b_n^{-1}\P(T_{k,n}\ge1) \to \alpha_k  \ \ \ n\to\infty, 
$$
where $\alpha_k$ is given in Corollary \ref{c:LDP.application}. \\
$(ii)$ If $a_n=o(n^{1/3})$, then 
$$
b_n^{-1}\P(T_{k,n}^\ms{B}\ge1) \to \alpha_k,  \ \ \ n\to\infty. 
$$
\end{corollary}
\medskip

\section{Proofs}  \label{sec:proofs}

\subsection{Proofs of Theorem \ref{t:LDP.L.kn}, Corollary \ref{c:LDP.LknB}, and Corollary \ref{c:LDP.application}}  \label{sec:proof.LDP}

First, let us generalize the radius \eqref{e:def.rn.s0} by 
$$
r_n(u) := \Big(\frac{a_n+u}{n\theta_d}\Big)^{1/d}, \  \ u\in \R. 
$$

For $x\in [0,1]^d$ and $\omega \in M_p\big( [0,1]^d \big)$, define the following functions: 
\begin{align*}
f(x,\omega) &:=  n\th_dR_k(x, \omega)^d-a_n, \quad\text{ and }\quad	g(x,\omega) := \one \{ f(x, \omega)>s_0 \}.
\end{align*}

{Then, the process \eqref{e:def.L.kn} that is restricted here to the space $M_p(E_0)$, can be reformulated as}
\begin{equation}  \label{e:def.Lkn2}
	L_{k,n} = \begin{cases}
	 \sum_{X\in \Pn}  g(X, \Pn)\, \delta_{(X, f(X,\Pn))} & \text{ if } |\Pn| > k, \\
	 \emptyset & \text{ if } |\Pn| \le k. 
	 \end{cases}
\end{equation}
Next, the unit cube $[0,1]^d$ is partitioned into smaller cubes $Q_1,\dots,Q_{b_n}$, so that $\Leb(Q_\ell)=b_n^{-1}$ for all $\ell\in\{ 1,\dots,b_n \}$. 
Here, it is assumed, without loss of generality, that $b_n$ takes only positive integers for all $n\ge1$. To avoid unnecessary technicalities, we will put the same assumption on many of the sequences and functions throughout  the proof. 
Fix  a sequence $w_n\to \infty$ with $w_n=o(a_n)$ as $n\to\infty$. For each $\ell\in\{  1,\dots,b_n \}$, define the ``boundary part" of $Q_\ell$ by
$$
M_\ell :=  \big\{ x\in Q_\ell: \inf_{y\in \partial Q_\ell} \| x-y\|\le r_n(w_n) \big\}, 
$$
while the ``internal region" of $Q_\ell$ is  given  as $K_\ell := Q_\ell \setminus M_\ell$. 

We now consider the process 
\begin{equation}  \label{e:def.eta.kn}
\eta_{k,n} := \sum_{\ell=1}^{b_n} \sum_{X\in \Pnr_{K_\ell} }g\big(X, \Pnr_{Q_\ell}\big) \, \delta_{(X, f(X, \Pnr_{Q_\ell}))}  \in M_p(\ois ), 
\end{equation}
where $\Pnr_{K_\ell}$ (resp.~$\Pnr_{Q_\ell}$) represents the Poisson point process restricted to $K_\ell$ (resp.~$Q_\ell$). Setting up  a ``blocked" point process as in \eqref{e:def.eta.kn} is  a standard approach  in the literature (see, e.g., \cite{seppalainen:yukich:2001, schreiber:yukich:2005}). Clearly, the process \eqref{e:def.eta.kn} is different from \eqref{e:def.L.kn}. For example, \eqref{e:def.eta.kn} removes all the points $X\in \Pn$ lying within the distance of $r_n(w_n)$ from the boundary of $(Q_\ell)_{\ell=1}^{b_n}$, whereas those points are possibly counted by \eqref{e:def.L.kn}. Even when the center $X\in \Pn$ is chosen from the inside of $K_\ell$, i.e., $X\in \Pnr_{K_\ell}$,  the processes \eqref{e:def.L.kn} and  \eqref{e:def.eta.kn} may  exhibit different $k$-nearest neighbor balls, whenever  $R_k(X,\Pn)\neq R_k(X, \Pnr_{Q_\ell})$. Despite such differences, it is justified later in  Propositions \ref{p:exp.neg.eta.etap} -- \ref{p:exp.neg.etap.Lp}  that the process  \eqref{e:def.eta.kn} can be used to approximate the large deviation behavior of  \eqref{e:def.L.kn}. 

For later analyses, it is convenient to express \eqref{e:def.eta.kn} as a superposition of i.i.d.~point processes on $E_0$,  which themselves are transformed by  some  homeomorphisms. 
For each $\ell\in\{ 1,\dots,b_n \}$, let $\kln: [0,1]^d\to Q_\ell$ be the homeomorphism defined by $\kln (x)=b_n^{-1/d}x+z_{\ell,n}$, where  $z_{\ell, n}$ is  a lower-left corner of $Q_\ell$. Further, define $\ktln: \ois \to Q_\ell \times (s_0,\infty]$ by $\ktln (x,u)=\big( \kln(x),u \big)$.

Using the homeomorphism $\ktln$, one can express $\eta_{k,n}$ as 
\begin{equation}  \label{e:def.eta.kn2}
\eta_{k,n} = \sum_{\ell=1}^{b_n} \eta_{k,n}^{(\ell)} \circ \ktln^{-1}, 
\end{equation}
where 
$$
\eta_{k,n}^{(\ell)} := \sum_{X\in \Pnr_{K_\ell}} g\big(X, \Pnr_{Q_\ell}\big) \, \delta_{(\kln^{-1}(X),\, f(X, \Pnr_{Q_\ell}))} \in M_p(\ois). 
$$
Due to the spatial independence and homogeneity of $\Pn$, $(\eta_{k,n}^{(\ell)})_{\ell\ge 1}$ constitutes a sequence of i.i.d.~point processes on $\ois$.  

Next, let $(\zeta_k^{(\ell)})_{\ell \ge 1}$ be a collection of i.i.d.~Poisson point processes on $\ois$ with intensity $\Leb \otimes \tau_k$.
Then, the proposition below claims that for each $\ell\ge1$, the law of $\eta_{k,n}^{(\ell)}$ converges to the law of $\zeta_k^{(\ell)}$ as $n\to\infty$, in terms of the \emph{Kantorovich-Rubinstein distance}. Recall that the Kantorovich-Rubinstein distance between the distributions of two point processes $\xi_i$, $i=1,2$, is defined as 
\begin{equation}  \label{e:def.KR.dist}
d_{\ms{KR}} \big( \mathcal L(\xi_1), \mathcal L(\xi_2)\big) := \sup_{h} \big| \E\big[ h(\xi_1) \big] - \E\big[ h(\xi_2) \big]  \big|, 
\end{equation}
where $\mathcal L(\xi_i)$ is a probability law of $\xi_i$, and 
$h$ is taken over all measurable $1$-Lipschitz functions with respect to the total variation distance on the space of point measures; see  \cite{bobrowski:schulte:yogeshwaran:2021} for more information on the Kantorovich-Rubinstein distance. As a related notion, the \emph{total variation distance} between two measures $\mu_1$ and $\mu_2$ on $E_0$ is defined as 
$$
d_{\ms{TV}} (\mu_1, \mu_2) :=\sup_{A\subset E_0} \big| \mu_1(A)-\mu_2(A) \big|. 
$$

As a final remark, we propose one key observation:  for $x\in [0,1]^d$, $\omega \in M_p\big( [0,1]^d \big)$ with $x\in \omega$, and $u\in \R$, the following conditions are equivalent.  
\begin{equation}  \label{e:K.and.Pn1}
f(x,\omega) >u \ \ \Leftrightarrow \ \ R_k(x,\omega) >r_n(u)\ \ \Leftrightarrow \ \ \omega \big( B_{r_n(u)}(x) \big) \le k. 
\end{equation}

Throughout the proof, $C^*$ denotes a generic, positive constant, which is independent of $n$ but may vary from one line to another or even within the lines. 

\begin{proposition}  \label{p:KR.conv1}
For every $\ell\ge1$, 
\begin{equation}  \label{e:KR.dist.eta.zeta}
d_{\ms{KR}} \big( \mathcal L(\eta_{k,n}^{(\ell)}), \mathcal L(\zeta_k^{(\ell)})\big) \to 0, \ \ \text{as } n\to\infty.  
\end{equation}
\end{proposition}
\begin{proof}
The proof is based on \cite[Theorem 6.4]{bobrowski:schulte:yogeshwaran:2021}. Before applying this theorem, we need some preliminary works. First, define for $x\in [0,1]^d$ and   $\omega\in M_p\big([0,1]^d\big)$, 
$$
\mathcal S(x,\omega) := B_{R_k(x,\omega)}(x); 
$$
then, $f$ and $g$ are \emph{localized} to $\mathcal S$. Namely, for every $x\in \omega$ and all $S\supset \mathcal S(x,\omega)$, we have $g(x,\omega)=g(x,\omega\cap S)$, and  also,  $f(x,\omega) = f(x,\omega\cap S)$ if $g(x,\omega)=1$. Moreover, $\mathcal S(x,\omega)$ is a \emph{stopping set}; that is, for every compact $S\subset [0,1]^d$, 
$$
\big\{ \omega: B_{R_k(x,\omega)}(x)\subset S  \big\} = \big\{ \omega: B_{R_k(x,\omega\cap S)}(x)\subset S  \big\}. 
$$
Finally, we set $S_x :=B_{r_n(w_n)}(x)$ for $x\in \QmM$.

According to   \cite[Theorem 6.4]{bobrowski:schulte:yogeshwaran:2021},  \eqref{e:KR.dist.eta.zeta} can be obtained as a direct consequence of the following conditions. First, one needs to show that 
\begin{equation}  \label{e:TV.dist.mean.measures}
d_{\ms{TV}} \Big(\E \big[ \eta_{k,n}^{(\ell)}(\cdot) \big], \, \Leb \otimes \tau_k\Big) \to 0, \ \ \ n\to\infty,
\end{equation}
where $\E \big[ \eta_{k,n}^{(\ell)}(\cdot) \big]$ denotes  the intensity measure of $\eta_{k,n}^{(\ell)}$. 
In addition to \eqref{e:TV.dist.mean.measures}, we also have to show that  as $n\to\infty$, 
\begin{equation}  \label{e:E1}
E_1 := 2n\int_{\QmM} \E \big[  g(x,\Pnr_{Q_\ell}+\delta_x) \, \one \big\{  \mathcal S (x,\Pnr_{Q_\ell}+\delta_x) \not\subset S_x \big\} \big]\dif x\to0, 
\end{equation}
\begin{equation}  \label{e:E2}
E_2 := 2n^2 \int_{\QmM}\int_{\QmM}  \hspace{-.2cm}\one \{ S_x \cap S_z \neq \emptyset \} \,\E \big[  g(x,\Pnr_{Q_\ell}+\delta_x) \big]  \E \big[  g(z,\Pnr_{Q_\ell}+\delta_z) \big] \dif x \dif z \to 0, 
\end{equation}
and 
\begin{equation}  \label{e:E3}
	E_3 := 2n^2 \int_{\QmM}\int_{\QmM} \hspace{-.2cm}\one \{ S_x \cap S_z \neq \emptyset \} \,\E \big[  g(x,\Pnr_{Q_\ell}+\delta_x+\delta_z) \,  g(z,\Pnr_{Q_\ell}+\delta_x+\delta_z) \big] \dif x \dif z \to 0. 
\end{equation}
Our goal in the sequel  is  to prove \eqref{e:TV.dist.mean.measures} -- \eqref{e:E3}. 
\medskip

\textit{Proof of \eqref{e:TV.dist.mean.measures}}: We begin with  calculating   the measure $\E \big[ \eta_{k,n}^{(\ell)}(\cdot) \big]$ more explicitly. 
	For $A \subset \oi$ and $u>s_0$, by \eqref{e:K.and.Pn1} and the Mecke formula for Poisson point processes (see, e.g., Chapter 4 in \cite{last:penrose:2017}), together with the fact that $\Pn(Q_\ell)$ is Poisson distributed with mean $n\Leb(Q_\ell)=nb_n^{-1}$, 
\begin{align*}
	 \E\big[ \eta_{k,n}^{(\ell)}\big(A \ti (u,\infty)\big) \big] &= \E\Big[ \sum_{X\in \Pnr_{\QmM}} \one \big\{\k_{\ell, n}^{-1}(X) \in A, \, f(X, \Pnr_{Q_\ell}) >u\big\} \Big] \\
&=\E\Big[ \sum_{X\in\Pnr_{\QmM}} \one \big\{ \k_{\ell, n}^{-1}(X) \in A,\,\Pnr_{Q_\ell} \big( B_{r_n(u)}(X) \big) \le k\big\}  \Big] \\
&=\E\Big[ \sum_{X\in\Pnr_{\QmM}} \one \big\{ \k_{\ell, n}^{-1}(X) \in A,\,\Pn \big( B_{r_n(u)}(X) \big) \le k\big\}  \Big] \\
	&=nb_n^{-1}\P \Big( \kln^{-1} (Y) \in A\setminus \kln^{-1}(M_\ell), \, (\Pn +\delta_Y) \big( B_{r_n(u)}(Y) \big)\le k \Big), 
\end{align*}
where $Y$ is a uniform random variable on $[0,1]^d$, independent of $\Pn$. 
At the third equality above, we have dropped the restriction of $\Pn$, i.e., $\Pnr_{Q_\ell}=\Pn$, because  $B_{r_n(w_n)}(X)\subset Q_\ell$ for all $X\in \Pnr_{K_\ell}$. 
By the  conditioning on $Y$, the last expression equals 
\begin{align*}
&nb_n^{-1}\E \Big[  \one \big\{  \kln^{-1}(Y)\in A \setminus \kln^{-1}(M_\ell) \big\}\, \P \big( \Pn (B_{r_n(u)}(Y))\le k-1 \, \big|\, Y \big) \Big]  \\
&=nb_n^{-1} \Leb \big(  A \setminus \kln^{-1}(M_\ell)  \big)\sum_{i=0}^{k-1} e^{-(a_n+u)} \frac{(a_n+u)^i}{i!}. 
\end{align*}
It thus turns out that $\E \big[ \eta_{k,n}^{(\ell)}(\cdot) \big]$ has the density 
$$
q_k(x,u) := nb_n^{-1} \one \big\{ x\notin \kln^{-1}(M_\ell) \big\}\, \frac{e^{-(a_n+u)} (a_n+u)^{k-1}}{(k-1)!}, \ \ x\in [0,1]^d, \, u >s_0. 
$$
Therefore, we have 
\begin{align*}
	d_{\ms{TV}} \Big(\E \big[ \eta_{k,n}^{(\ell)}(\cdot) \big], \, \Leb \otimes \tau_k  \Big) &\le \int_{E_0} \big|\,  q_k(x,u)-\frac{e^{-u}}{(k-1)!} \, \big| \d x \d u \\
&\le \Leb \big( \kln^{-1}(M_\ell) \big)\, \frac{e^{-s_0}}{(k-1)!} + \int_{s_0}^\infty \Big| \,  \Big(1+\frac{u}{a_n}  \Big)^{k-1}-1  \Big| \, e^{-u}\d u. 
\end{align*}
The first  term above tends to $0$ as $n\to\infty$, because 
$$
\Leb \big( \kln^{-1}(M_\ell) \big) =b_n \Leb (M_\ell) = 1-\big( 1-b_n^{1/d} r_n(w_n) \big)^d \to 0, \ \ \ n\to\infty, 
$$
while the second term  vanishes by the dominated convergence theorem. 
\medskip

\textit{Proof of \eqref{e:E1}}: It follows from \eqref{e:K.and.Pn1} that 
\begin{align*}
\one \big\{ \mathcal S (x,\Pnr_{Q_\ell}+\delta_x) \not\subset S_x \big\} &= \one \big\{ B_{R_k(x, \Pnr_{Q_\ell}+\delta_x)}(x)  \supset B_{r_n(w_n)}(x) \big\} \\
&= \one \big\{  (\Pnr_{Q_\ell}+\delta_x) \big( B_{r_n(w_n)}(x) \big)\le k \big\}\\
&= \one \big\{ \Pnr_{Q_\ell}\big( B_{r_n(w_n)}(x) \big) \le k-1 \big\}\\
&= \one \big\{ \Pn\big( B_{r_n(w_n)}(x) \big) \le k-1 \big\}. 
\end{align*}
At the fourth equality above, we have  dropped the restriction of $\Pn$, due to the fact that $B_{r_n(w_n)}(x)\subset Q_\ell$ for all $x\in K_\ell$. 
Now, as $n\to\infty$, 
\begin{align*}
E_1 &\le 2n \int_{\QmM} \P \Big( \Pn \big( B_{r_n(w_n)} (x)\big) \le k-1 \Big)\dif x =2n \,\Leb(\QmM) \sum_{i=0}^{k-1} e^{-(a_n+w_n)} \frac{(a_n+w_n)^i}{i!} \\
&\le 2e^{-w_n} \sum_{i=0}^{k-1} \frac{(a_n+w_n)^i}{i! a_n^{k-1}} \le C^*e^{-w_n}\to 0.  
\end{align*}
\medskip

\textit{Proof of \eqref{e:E2}}: One can see that for $x,z\in K_\ell$,
\begin{equation}  \label{e:Sx.Sz.indicator}
\one \{ S_x \cap S_z \neq \emptyset \} \le \one \big\{ \| x-z\| \le 2r_n(w_n) \big\}, 
\end{equation}
and from \eqref{e:K.and.Pn1}, 
$$
\E \big[ g(x,\Pnr_{Q_\ell}+\delta_x) \big] = \P \Big( \Pn \big( B_{r_n(s_0)} (x)\big) \le k-1 \Big) =\sum_{i=0}^{k-1} e^{-(a_n+s_0)}\frac{(a_n+s_0)^i}{i!}. 
$$
Therefore, as $n\to\infty$, 
\begin{align*}
E_2 &\le 2n^2 \Big\{  \sum_{i=0}^{k-1} e^{-(a_n+s_0)}\frac{(a_n+s_0)^i}{i!}\Big\}^2 \int_{Q_\ell}\int_{Q_\ell} \one \big\{ \| x-z\| \le 2r_n(w_n) \big\} \dif x \dif z \\
&\le C^*b_n^2  \int_{Q_\ell}\int_{Q_\ell} \one \big\{ \| x-z\| \le 2r_n(w_n) \big\} \dif x \dif z \\
&\le C^*b_n^2 \Leb(Q_\ell) r_n(w_n)^d = C^*b_nr_n(w_n)^d\to 0. 
\end{align*}
\medskip

\textit{Proof of \eqref{e:E3}}: It follows from \eqref{e:K.and.Pn1}  and \eqref{e:Sx.Sz.indicator} that   $E_3$ can be split into two terms: 
\begin{equation}  \label{e:E3.split}
\begin{split}
E_3 &\le 2n^2 \int_{\QmM}\int_{\QmM}\one \big\{ \| x-z\| \le 2r_n(w_n) \big\}  \\
&\qquad \qquad \qquad \times \P \Big( (\Pn+\delta_z) \big( B_{r_n(s_0)}(x) \big) \le k -1, \,   (\Pn+\delta_x) \big( B_{r_n(s_0)}(z) \big) \le k -1\Big) \dif x \dif z \\
&= 2n^2 \int_{\QmM}\int_{\QmM}\one \big\{ \| x-z\| \le r_n(s_0) \big\}  \\
&\qquad \qquad \qquad \times \P \Big( (\Pn+\delta_z) \big( B_{r_n(s_0)}(x) \big) \le k -1, \,   (\Pn+\delta_x) \big( B_{r_n(s_0)}(z) \big) \le k -1\Big) \dif x \dif z \\
&\quad + 2n^2 \int_{\QmM}\int_{\QmM}\one \big\{r_n(s_0) < \| x-z\| \le 2r_n(w_n) \big\}  \\
&\qquad \qquad \qquad \times \P \Big( (\Pn+\delta_z) \big( B_{r_n(s_0)}(x) \big) \le k -1, \,   (\Pn+\delta_x) \big( B_{r_n(s_0)}(z) \big) \le k -1\Big) \dif x \dif z \\
&=: E_{3,1}+E_{3,2}. 
\end{split}
\end{equation}
For $E_{3,1}$, if $\| x-z\|\le r_n(s_0)$ with $x,z \in \QmM$, then $(\Pn+\delta_z) \big( B_{r_n(s_0)}(x) \big) = \Pn \big( B_{r_n(s_0)}(x) \big)+1$. Because of the spatial independence of $\Pn$, 
\begin{align}
E_{3,1} &\le 2n^2 \int_{\QmM}\int_{\QmM} \one \big\{ \| x-z \|\le r_n(s_0) \big\}\, \P\Big(\Pn \big( B_{r_n(s_0)}(x) \big) \le k-2  \Big)  \label{e:E31} \\
&\qquad \qquad \qquad \qquad \qquad   \times \P\Big(\Pn \big( B_{r_n(s_0)}(z)\setminus B_{r_n(s_0)}(x) \big) \le k-2  \Big)\dif x \dif z. \notag 
\end{align}
Then, it is easy to see that 
$$
\P\Big(\Pn \big( B_{r_n(s_0)}(x) \big) \le k-2  \Big) \le C^*a_n^{k-2}e^{-a_n}, 
$$
and also, 
$$
\P\Big(\Pn \big( B_{r_n(s_0)}(z)\setminus B_{r_n(s_0)}(x) \big) \le k-2  \Big) \le C^*e^{-\frac{n}{2}\Leb(B_{r_n(s_0)}(z)\setminus B_{r_n(s_0)}(x))}. 
$$
Notice that 
$$
\Leb\big(B_{r_n(s_0)}(z)\setminus B_{r_n(s_0)}(x)\big) \ge C^*r_n(s_0)^{d-1}\| x-z \|, 
$$
whenever $\|x-z\| \le 2r_n(s_0)$ (see Equ.~(7.5) in \cite{penrose:goldstein:2010}); so, we have
$$
\P\Big(\Pn \big( B_{r_n(s_0)}(z)\setminus B_{r_n(s_0)}(x) \big) \le k-2  \Big) \le C^*e^{-\frac{n}{2}C^*r_n(s_0)^{d-1}\| x-z \|}. 
$$
Referring these bounds back to \eqref{e:E31}, 
\begin{align*}
E_{3,1} &\le C^*n^2 \int_{Q_\ell}\int_{\R^d}a_n^{k-2}e^{-a_n} e^{-\frac{n}{2}C^*r_n(s_0)^{d-1}\| x-z \|} \dif x \dif z \\
&=C^*n^2a_n^{k-2}e^{-a_n} b_n^{-1} \int_0^\infty e^{-\frac{n}{2}C^*r_n(s_0)^{d-1}\rho} \rho^{d-1}\dif \rho \\
&=  \frac{C^*n^2a_n^{k-2}e^{-a_n} b_n^{-1} }{\big( nr_n(s_0)^{d-1} \big)^d} = C^*\Big( 1+\frac{s_0}{a_n} \Big)\, \frac{1}{(a_n+s_0)^d} \to 0, \ \ \ n\to\infty. 
\end{align*}
By the spatial independence of $\Pn$, 
\begin{align*}
E_{3,2} &\le 2n^2 \int_{\QmM}\int_{\QmM} \P\Big( \Pn\big( B_{r_n(s_0)}(x) \big)\le k-1 \Big) \P\Big( \Pn\big( B_{r_n(s_0)}(z)\setminus B_{r_n(s_0)}(x)   \big)\le k-1 \Big) \\
&\qquad \qquad \qquad \qquad \qquad \qquad \times \one \big\{ r_n(s_0) < \| x-z\| \le 2r_n(w_n) \big\} \dif x \dif z 
\end{align*}
Then, 
$$
\P\Big( \Pn\big( B_{r_n(s_0)}(x) \big)\le k-1 \Big) \le C^*a_n^{k-1} e^{-a_n}. 
$$
Moreover, if $\| x-z \|>r_n(s_0)$ with $x,z\in \QmM$, then 
$$
n\, \Leb \big(  B_{r_n(s_0)}(z)\setminus B_{r_n(s_0)}(x)  \big) \ge \frac{n}{2}\, \Leb \big( B_{r_n(s_0)}(z) \big) = \frac{a_n+s_0}{2}, 
$$
from which we have 
$$
\P\Big( \Pn\big( B_{r_n(s_0)}(z)\setminus B_{r_n(s_0)}(x)   \big)\le k-1 \Big) \le \sum_{i=0}^{k-1} e^{-\frac{a_n+s_0}{2}} \frac{(a_n+s_0)^i}{i!} \le C^*a_n^{k-1}e^{-\frac{a_n}{2}}. 
$$
Appealing to these obtained bounds, as $n\to\infty$, 
\begin{align*}
E_{3,2} &\le C^*n^2\int_{Q_\ell}\int_{\R^d} a_n^{k-1}e^{-a_n} \cdot a_n^{k-1} e^{-\frac{a_n}{2}} \one \big\{  \|x-z\|\le 2r_n(w_n) \big\}\dif x \dif z \\
&=C^* n^2 a_n^{2(k-1)} e^{-\frac{3a_n}{2}} b_n^{-1} r_n(w_n)^d \le C^* a_n^k e^{-\frac{a_n}{2}} \to 0, 
\end{align*}
as desired. 
\end{proof}

Recall now that $(\zeta_k^{(\ell)})_{\ell\ge1}$ are i.i.d.~Poisson point processes on $\ois$ with intensity $\Leb \otimes \tau_k$. 
The next proposition claims that the process 
\begin{equation}  \label{e:def.zeta.kn}
	\zeta_{k,n} := \sum_{\ell=1}^{b_n} \zeta_k^{(\ell)} \circ \ktln^{-1}
\end{equation}
satisfies the desired  LDP in Theorem \ref{t:LDP.L.kn}. Let $(\Omega', \mathcal F', \P')$ be the probability space for which \eqref{e:def.zeta.kn} is defined. 

\begin{proposition}  \label{p:LDP.zeta.kn}
The sequence $(\zeta_{k,n}/b_n)_{n\ge1}$ satisfies an LDP in the weak topology, with rate $b_n$ and rate function $H_k(\cdot\, | \, \Leb \otimes \tau_k) = \Lambda_k^*$.
\end{proposition}

\begin{proof}
By the transformation theorem, $\zeta_k^{(\ell)} \circ \ktln^{-1}$ becomes a Poisson point process on $\ois$ with intensity $(\Leb \otimes \tau_k) \circ \ktln^{-1}  =b_n (\Leb|_{Q_\ell}\otimes \tau_k)$. As $(\zeta_k^{(\ell)})_{\ell\ge1}$ are i.i.d., \eqref{e:def.zeta.kn} turns out to be a Poisson point process on $\ois$ with intensity $b_n (\Leb \otimes \tau_k)$; thus, there exists a sequence $(\xi_k^{(\ell)})_{\ell\ge1}$ of i.i.d.~Poisson point processes on $\ois$ with intensity $\Leb \otimes \tau_k$, so that 
\begin{equation}  \label{e:superposition.equality.dist}
\zeta_{k,n} \stackrel{d}{=}  \sum_{\ell=1}^{b_n} \xi_k^{(\ell)}. 
\end{equation}
For convenience, we assume that $(\xi_k^{(\ell)})_{\ell\ge1}$ are defined in the same probability space $(\Omega', \mathcal F', \P')$. 
By applying the Poisson variant of Sanov's theorem to \eqref{e:superposition.equality.dist}, we conclude that  $(\zeta_{k,n}/b_n)_{n\ge1}$ satisfies an LDP with rate $b_n$ and rate function $H_k(\cdot\, |\, \Leb \otimes \tau_k)$; see \cite[Proposition 3.6]{wireless3} and  \cite[Theorem 6.2.10]{dembo:zeitouni:1998} for details.
To deduce the LDP with rate function $\Lambda_k^*$, 
we compute the logarithmic Laplace functional of $\xi_k^{(1)}$: it follows from Theorem 5.1 in \cite{resnick:2007} that for every $f\in C_b(\ois)$, 
$$
\Lambda_k(f):= \log \E' \big[ e^{\xi_k^{(1)}}(f) \big] =\int_{\ois} \big( e^{f(x,u)}-1 \big)\,\d x \, \tau_k(\d u). 
$$
Now, Cram\'er's theorem in Polish spaces (see, e.g., Theorem 6.1.3 in \cite{dembo:zeitouni:1998}) can yield the required LDP, for which $\Lambda_k^*$ is obtained as the Legendre transform of $\Lambda_k$. 

\end{proof}

By the maximal coupling argument (see \cite[Lemma 4.32]{kallenberg}), for every $\ell\ge1$ there exists a coupling $(\hat \eta_{k,n}^{(\ell)}, \hat \zeta_k^{(\ell)})$ defined on a probability space $(\hat \Omega_\ell, \mathcal{\hat F}_\ell,  \hat \P_\ell)$, such that $\hat \eta_{k,n}^{(\ell)} \stackrel{d}{=} \eta_{k,n}^{(\ell)}$ and $\hat \zeta_k^{(\ell)} \stackrel{d}{=}  \zeta_k^{(\ell)}$, and 
\begin{equation}  \label{e:max.coupling.conv}
\hat \P_\ell \big(\hat \eta_{k,n}^{(\ell)} \neq \hat \zeta_k^{(\ell)} \big) = d_{\ms{TV}} \big( \mathcal L (\eta_{k,n}^{(\ell)}), \mathcal L (\zeta_k^{(\ell)}) \big) \le d_{\ms{KR}} \big( \mathcal L (\eta_{k,n}^{(\ell)}), \mathcal L (\zeta_k^{(\ell)}) \big) \to 0, \ \ \ n\to\infty. 
\end{equation}
In particular, $(\hat \eta_{k,n}^{(\ell)}, \hat \zeta_{k}^{(\ell)})_{\ell\ge1}$ constitutes a sequence of i.i.d.~random vectors on the probability space $(\hat \Omega, \hat{\mathcal F}, \hat \P)$, where $\hat \Omega = \prod_{\ell=1}^\infty \hat \Omega_\ell$, $\hat{\mathcal F} =  \bigotimes_{\ell=1}^\infty \hat{\mathcal F}_\ell$, and $\hat \P=  \bigotimes_{\ell=1}^\infty \hat \P_\ell$. 
Defining $\hat \eta_{k,n}$ and $\hat \zeta_{k,n}$ analogously to \eqref{e:def.eta.kn2} and \eqref{e:def.zeta.kn}, Proposition \ref{p:eta.zeta.negligible} below  demonstrates that $(\hat \eta_{k,n}/b_n)_{n\ge1}$ and $(\hat \zeta_{k,n}/b_n)_{n\ge1}$ are exponentially equivalent (in terms of the total variation distance) under the  coupled probability measure $\hat \P$. Since the LDP for $(\zeta_{k,n}/b_n)_{n\ge1}$ was already given by Proposition \ref{p:LDP.zeta.kn}, this   exponential equivalence allows us to conclude that $(\eta_{k,n}/b_n)_{n\ge1}$ fulfills an LDP in Theorem \ref{t:LDP.L.kn}. 

\begin{proposition}  \label{p:eta.zeta.negligible}
For every $\delta>0$, 
$$
\frac{1}{b_n}\log \hat \P \big( d_{\ms{TV}}( \hat\eta_{k,n}, \, \hat\zeta_{k,n}) \ge \delta b_n\big)\to -\infty, \ \ \ n\to\infty. 
$$
\end{proposition}

\begin{proof}
The proof is highly related to \cite[Lemma 5.5]{hirsch:owada:2022}, but we still want to give a concise and self-contained proof. 
By Markov's inequality and the fact that $(\hat \eta_{k,n}^{(\ell)}, \hat \zeta_{k}^{(\ell)})_{\ell\ge1}$ are i.i.d.~processes, we have, for every $a>0$, 
\begin{align*}
\frac{1}{b_n} \log \hat \P \big( d_{\ms{TV}} ( \hat \eta_{k,n},  \hat \zeta_{k,n}) \ge \delta b_n\big) &\le  \frac{1}{b_n} \log \hat \P  \Big( \sum_{\ell=1}^{b_n} d_\ms{TV} (\hat \eta_{k,n}^{(\ell)}, \hat \zeta_k^{(\ell)}) \ge \delta b_n \Big)  
\le -a\delta +\log \hat \E \big[e^{a d_\ms{TV} (\hat \eta_{k,n}^{(1)}, \, \hat \zeta_{k}^{(1)}) }  \big], 
\end{align*}
where $\hat \E$ denotes an expectation with respect to $(\hat \Omega, \hat{\mathcal F}, \hat \P)$. Since $a>0$ is arbitrary, the desired result immediately follows if we can prove that for every $a>0$, 
\begin{equation}  \label{e:dTV.exp.conv}
\lim_{n\to\infty}\hat \E \Big[e^{a d_\ms{TV} (\hat \eta_{k,n}^{(1)}, \, \hat \zeta_{k}^{(1)}) }  \Big] =1. 
\end{equation}
By virtue of \eqref{e:max.coupling.conv}, $d_\ms{TV}(\hat \eta_{k,n}^{(1)}, \hat \zeta_k^{(1)})$ converges to $0$ in probability with respect to $\hat \P$. By the Cauchy-Schwarz inequality, \eqref{e:K.and.Pn1}, and the fact that $\zeta_k^{(1)}(E_0)$ is Poisson with mean $e^{-s_0}/(k-1)!$, 
\begin{align*}
\hat \E \Big[e^{a d_\ms{TV} (\hat \eta_{k,n}^{(1)}, \, \hat \zeta_{k}^{(1)}) }  \Big] &\le \bigg\{ \E \Big[ e^{2a \eta_{k,n}^{(1)}(E_0)} \Big]  \bigg\}^{1/2} \bigg\{ \E' \Big[ e^{2a \zeta_k^{(1)}(E_0)} \Big] \bigg\}^{1/2} \\
&= \bigg\{ \E \Big[ e^{2a \sum_{X\in \Pnr_{K_1}} \one \big\{ \Pn(B_{r_n(s_0)}(X))\le k \big\} }\Big]  \bigg\}^{1/2}  \exp \Big\{  \frac{e^{-s_0}}{2(k-1)!} (e^{2a}-1) \Big\}. 
\end{align*}
Now, the desired uniform integrability for   \eqref{e:dTV.exp.conv} follows, provided that for every $a>0$, 
\begin{equation}  \label{e:UI1}
\limsup_{n\to\infty} \E \Big[ e^{a \sum_{X\in \Pnr_{K_1}} \one \big\{ \Pn(B_{r_n(s_0)}(X))\le k \big\} }\Big]  <\infty. 
\end{equation}
As in the proof of \cite[Lemma 5.5]{hirsch:owada:2022}, by utilizing the diluted family of cubes
$$
G=\big\{  4r_n(s_0)z + \big[ 0,r_n(s_0)/\sqrt{d} \big]^d\subset Q_1: z\in \bbz^d \big\}, 
$$
it turns out that  \eqref{e:UI1} is  obtained as a consequence of 
{
$$
\limsup_{n\to\infty} \bigg\{ \E \Big[ e^{a\sum_{X\in\Pn} \one \big\{ \Pn (B_{r_n(s_0)}(X)) \le k, \, X\in K_1 \cap J  \big\} }\Big] \bigg\}^{1/(b_n(4r_n(s_0))^d)} < \infty, 
$$
}
where $J=\big[ 0, r_n(s_0)/\sqrt{d} \big]^d$ {and $1/(b_n(4r_n(s_0))^d)$ represents the number of cubes in $G$}. Notice that 
\begin{equation}  \label{e:up.to.k}
\sum_{X\in\Pn} \one \big\{ \Pn (B_{r_n(s_0)}(X)) \le k, \, X\in K_1\cap J \big\}\in \{ 0,1,\dots,k \}, 
\end{equation}
because if there exist more than $k$  points inside $J$,  these points never  contribute to \eqref{e:up.to.k}. Therefore, by Markov's inequality, 
{
\begin{align}
&\bigg\{ \E \Big[ e^{a\sum_{X\in\Pn} \one \big\{ \Pn (B_{r_n(s_0)}(X)) \le k, \, X\in K_1 \cap J  \big\} }\Big] \bigg\}^{1/(b_n(4r_n(s_0))^d)}  \label{e:suff.UI1}\\
&\le \bigg( 1 + \sum_{\ell=1}^k e^{a\ell} \E \Big[ \sum_{X\in\Pn} \one \big\{ \Pn (B_{r_n(s_0)}(X)) \le k, \, X\in K_1 \cap J  \big\} \Big] \bigg)^{1/(b_n(4r_n(s_0))^d)}\notag  \\
&= \bigg( 1 + \sum_{\ell=1}^k e^{a\ell}  n \sum_{i=0}^{k-1} e^{-(a_n+s_0)} \frac{(a_n+s_0)^i}{i!} \P (X_1 \in  K_1 \cap J   )\bigg)^{1/(b_n(4r_n(s_0))^d)} \notag \\
&\le \big( 1+C^* e^{ak} b_nr_n(s_0)^d  \big)^{1/(b_n(4r_n(s_0))^d)}  \to e^{C^*e^{ak}/4^d} < \infty, \ \ \text{as } n\to\infty. \notag
\end{align} 
}
\end{proof}

\medskip

As shown in the last two propositions,  $(\eta_{k,n}/b_n)_{n\ge1}$  satisfies the LDP in Theorem \ref{t:LDP.L.kn}. Thus, our  final task is to demonstrate  that $(L_{k,n}/b_n)_{n\ge1}$ exhibits the same LDP as $(\eta_{k,n}/b_n)_{n\ge1}$. Although this can be done by 
  establishing  exponential equivalence between $(L_{k,n}/b_n)_{n\ge1}$ and $(\eta_{k,n}/b_n)_{n\ge1}$ in terms of the total variation distance, proving directly this  exponential equivalence seems to be difficult. Alternatively,  we set up   an additional sequence 
\begin{equation}  \label{e:def.eta.kn.p}
\eta_{k,n}' := \sum_{\ell=1}^{b_n} \sum_{X\in\Pnr_{K_\ell}}g\big(X, \Pnr_{Q_\ell}\big) \, \one \big\{ R_k(X, \Pnr_{Q_\ell}) \le r_n(w_n) \big\}  \, \delta_{(X, \, f(X, \Pnr_{Q_\ell}))}, 
\end{equation}
and prove  that $(\eta_{k,n}'/b_n)_{n\ge1}$ shows the same LDP as $(\eta_{k,n}/b_n)_{n\ge1}$ (see Proposition \ref{p:exp.neg.eta.etap}). Subsequently, we  define 
$$
L_{k,n}' := \begin{cases}
\sum_{X\in\Pn} g(X, \Pn)\, \one \big\{ R_k(X,\Pn)\le \sqrt{d}b_n^{-1/d} \big\}\, \delta_{(X, f(X, \Pn))} & \text{ if } |\Pn|>k,\\
\emptyset & \text{ if } |\Pn| \le k. 
\end{cases}
$$
and prove that $(L_{k,n}'/b_n)_{n\ge1}$ satisfies the same LDP as $(L_{k,n}/b_n)_{n\ge1}$ (see Proposition \ref{p:exp.neg.L.Lp}). Finally, Proposition \ref{p:exp.neg.etap.Lp} gives   exponential equivalence between $(L_{k,n}'/b_n)_{n\ge1}$ and  $(\eta_{k,n}'/b_n)_{n\ge1}$. Combining  Propositions \ref{p:exp.neg.eta.etap} -- \ref{p:exp.neg.etap.Lp} concludes  the required exponential equivalence between $(L_{k,n}/b_n)_{n\ge1}$ and  $(\eta_{k,n}/b_n)_{n\ge1}$.

\begin{proposition} \label{p:exp.neg.eta.etap}
The sequence $(\eta_{k,n}'/b_n)_{n\ge1}$ satisfies the same LDP as $(\eta_{k,n}/b_n)_{n\ge1}$. 
\end{proposition}

\begin{proof}
For our purpose, we demonstrate that for every $\delta>0$, 
\begin{equation}  \label{e:diff.eta.kn.and.eta,kn.prime}
\frac{1}{b_n} \log \P \big( d_{\ms{TV}}(\eta_{k,n}, \, \eta_{k,n}') \ge  \delta b_n\big) \to -\infty,  \ \ \ n\to\infty. 
\end{equation}
First, we see from \eqref{e:K.and.Pn1} that 
\begin{align*}
d_{\ms{TV}}(\eta_{k,n}, \, \eta_{k,n}') &\le  \sum_{\ell=1}^{b_n} \sum_{X\in \Pnr_{\QmM}} \one \big\{ R_k(X, \Pnr_{Q_\ell}) >r_n(w_n) \big\} \\
&=  \sum_{\ell=1}^{b_n} \sum_{X\in \Pnr_{\QmM}} \one \Big\{  \Pn \big( B_{r_n(w_n)}(X) \big) \le k  \Big\}, 
\end{align*}
from which one can bound \eqref{e:diff.eta.kn.and.eta,kn.prime} by 
$$
-a\delta + \log \E\Big[ e^{a\sum_{X\in \Pnr_{K_1}} \one \{ \Pn( B_{r_n(w_n)}(X) ) \le k \} } \Big]
$$
for every $a>0$. 
Now, one has to show that for every $a>0$, 
\begin{equation}  \label{e:UI2}
\limsup_{n\to\infty} \E \Big[ e^{a\sum_{X\in\Pnr_{K_1}} \one \{ \Pn( B_{r_n(w_n)}(X) ) \le k \} } \Big] \le 1. 
\end{equation}
As in the proof of Proposition \ref{p:eta.zeta.negligible}, \eqref{e:UI2} can be implied by 
$$
\limsup_{n\to\infty} \bigg\{ \E \Big[ e^{a\sum_{X\in\Pn} \one \big\{ \Pn (B_{r_n(w_n)}(X)) \le k, \, X\in K_1 \cap J  \big\} }\Big] \bigg\}^{1/(b_nr_n(w_n)^d)} \le 1, 
$$
for all $a>0$, where $J=\big[ 0, r_n(w_n)/\sqrt{d} \big]^d$. Now, instead of \eqref{e:suff.UI1}, we have that 
\begin{align*}
&\limsup_{n\to\infty} \bigg\{ \E \Big[ e^{a\sum_{X\in\Pn} \one \big\{ \Pn (B_{r_n(w_n)}(X)) \le k, \, X\in K_1 \cap J  \big\} }\Big] \bigg\}^{1/(b_nr_n(w_n)^d)} \\
&\le \limsup_{n\to\infty} \big( 1+C^* e^{ak} b_n r_n(w_n)^d e^{-w_n} \big)^{1/(b_nr_n(w_n)^d)} \\
&=\limsup_{n\to\infty} e^{C^*e^{ak} e^{-w_n}} = 1, 
\end{align*}
as required. 
\end{proof}

\begin{proposition}  \label{p:exp.neg.L.Lp}
The sequence $(L_{k,n}'/b_n)_{n\ge1}$ satisfies the same LDP as $(L_{k,n}/b_n)_{n\ge1}$. 
\end{proposition}

\begin{proof}
Throughout the proof, we assume $|\Pn|>k$. Using the bound 
$$
d_{\ms{TV}} (L_{k,n},\, L_{k,n}') \le \sum_{X\in\Pn} \one \big\{ R_k(X,\Pn) > \sqrt{d}b_n^{-1/d} \big\}, 
$$
it is sufficient to show that,  for every $\delta>0$, 
$$
\frac{1}{b_n} \log \P \Big( \sum_{X\in\Pn} \one \big\{ R_k(X,\Pn) >\sqrt{d} b_n^{-1/d} \big\} \ge \delta b_n \Big) \to -\infty,  \ \ \ n\to\infty. 
$$
Suppose there exists a point $X \in \Pn \cap Q_\ell$ for some $\ell\in \{ 1,\dots,b_n\}$ so that $R_k(X, \Pn) > \sqrt{d}b_n^{-1/d}$. Then, $B_{R_k(X,\Pn)}(X)\cap Q_\ell$ contains at most $k+1$ points of $\Pn$ (including $X$ itself). Thus,
	
\begin{align*}
\sum_{X\in\Pn} \one \big\{ R_k(X,\Pn) >\sqrt db_n^{-1/d} \big\} &= \sum_{\ell=1}^{b_n} \sum_{X \in \Pnr_{ Q_\ell}} \one \big\{ R_k(X,\Pn) >\sqrt{d}b_n^{-1/d} \big\} \\
&\le (k+1) \sum_{\ell=1}^{b_n} \one\big\{ \Pn (Q_\ell) \le k+1\big\}. 
\end{align*}
Now, we only have to demonstrate that 
\begin{equation}  \label{e:at.most.k+1.smaller.cube}
\frac{1}{b_n} \log \P \Big( \sum_{\ell=1}^{b_n} \one\big\{ \Pn (Q_\ell) \le k+1\big\} \ge \delta b_n\Big) \to -\infty, \ \ \ n\to\infty. 
\end{equation}
By Markov's inequality, we have, for every $a>0$, 
\begin{align*}
\frac{1}{b_n} \log \P \Big( \sum_{\ell=1}^{b_n} \one\big\{ \Pn (Q_\ell) \le k+1\big\} \ge \delta b_n\Big) &\le -a\delta + \log \E \Big[ e^{a\one \{ \Pn (Q_1) \le k+1 \}} \Big] \\
&\le -a\delta + \log \Big( 1 + e^a \sum_{i=0}^{k+1} e^{-n\Leb(Q_1)}\, \frac{(n\Leb(Q_1))^i}{i!} \Big) \\
	&\le -a\delta + \log \Big( 1 +  (k+2)e^ae^{-nb_n^{-1}} (nb_n^{-1})^{k+1}  \Big) \\
&\to -a\delta,  \ \ \text{as } n\to\infty. 
\end{align*}
As $a>0$ is arbitrary, we have obtained \eqref{e:at.most.k+1.smaller.cube}. 
\end{proof}

\begin{proposition}  \label{p:exp.neg.etap.Lp}
The sequence $(L_{k,n}'/b_n)_{n\ge1}$ exhibits the same LDP as $(\eta_{k,n}'/b_n)_{n\ge1}$. 
\end{proposition}

\begin{proof}
We  prove exponential equivalence between the two sequences: for every $\delta>0$, 

$$
\frac{1}{b_n} \log \P \big( d_{\ms{TV}}(L_{k,n}', \, \eta_{k,n}') \ge \delta  b_n\big) \to -\infty,  \ \ \ n\to\infty. 
$$
For convenience, let us slightly change the formulation of $\eta_{k,n}'$ given at \eqref{e:def.eta.kn.p}. To begin, observe that    if $X\in \Pnr_{K_\ell}$ with $R_k(X, \Pnr_{Q_\ell})\le r_n(w_n)$, then  $R_k(X, \Pnr_{Q_\ell}) = R_k(X,\Pn)$. Hence, one can replace the restricted process $\Pnr_{Q_\ell}$ in \eqref{e:def.eta.kn.p} with $\Pn$; that is, 
$$
\eta_{k,n}'= \sum_{\ell=1}^{b_n} \sum_{X\in\Pnr_{K_\ell} }g (X, \Pn ) \, \one \big\{ R_k(X, \Pn) \le r_n(w_n) \big\}  \, \delta_{(X, \, f(X, \Pn))}. 
$$
Using this representation and assuming  $|\Pn|>k$ allows us to express the total variation distance  in such a way that 
\begin{align}
	d_{\ms{TV}}(L_{k,n}', \,\eta_{k,n}') &= \frac{1}{b_n} \sup_{A\subset \ois} \Big\{  \sum_{X\in\Pn} \one \big\{  (X, f(X,\Pn))\in A, \, R_k(X,\Pn) \le \sqrt{d}b_n^{-1/d} \big\} \label{e:dTV.boundary.bound}\\
	&\qquad  - \sum_{\ell=1}^{b_n} \sum_{X\in\Pnr_{\QmM}} \one \big\{  (X, f(X,\Pn))\in A, \, R_k(X,\Pn) \le r_n(w_n) \big\}  \, \Big\} \notag \\
&=: \frac{1}{b_n} \sup_{A\subset \ois} \big( T_n^{(1)} -  T_n^{(2)}\big). \notag 
\end{align}
We now derive an upper bound of \eqref{e:dTV.boundary.bound}. We  consider a $k$-nearest neighbor ball  centered at $X\in \Pn$  with $(X,f(X,\Pn))\in A$.  Suppose  this ball is counted  by $T_n^{(1)}$, but not counted by $T_n^{(2)}$. Then, this ball must be of either Type 1 or Type 2 as defined below. 
\vspace{7pt}

\noindent \textbf{Type 1}: The center $X$ is located in $Q_\ell$ for some $\ell \in \{ 1,\dots,b_n \}$ such that 
$$
r_n(w_n) <R_k(X,\Pn)\le  \sqrt{d} b_n^{-1/d}. 
$$
\noindent \textbf{Type 2}: The center $X$ is in $Q_\ell^\partial (r_n(w_n))$ for some $\ell\in \{ 1,\dots,b_n \}$, where  
$$
Q_\ell^\partial(r) := \big\{ x\in Q_\ell: \inf_{y\in \partial Q_\ell} \| x-y\| \le r \big\},  \ \ r>0, 
$$
so that $r_n(s_0) <R_k(X,\Pn) \le r_n(w_n)$. 
\vspace{7pt}

\noindent Then, the number of $k$-nearest neighbor balls of Type 1 can be bounded by 
$$
U_n^{(1)} :=  \sum_{\ell=1}^{b_n} \sum_{X\in \Pnr_{Q_\ell}} \one \big\{ R_k(X,\Pn)\in \big( r_n(w_n), \sqrt{d}b_n^{-1/d} \big] \big\}, 
$$
while the number of those of Type 2 is bounded by 
$$
U_n^{(2)} :=  \sum_{\ell=1}^{b_n}  \sum_{X\in \Pnr_{Q_\ell^\partial (r_n(w_n))}} \one \big\{ R_k(X,\Pn)\in \big( r_n(s_0), r_n(w_n) \big] \big\}. 
$$
Hence,  the desired result will follow if one can show that 
for every $\delta>0$ and $j=1,2$,
\begin{equation}  \label{e:Unj.negligibility}
\frac{1}{b_n} \log \P (U_n^{(j)} \ge \delta b_n) \to -\infty,  \ \ \ n\to\infty.
\end{equation}
We first deal  with the case $j=1$. Let $S_\ell$ be the collection of cubes $(Q_i)_{i=1}^{b_n}$ that intersect with 
$$
\text{Tube} (Q_\ell, \sqrt{d}b_n^{-1/d}) := \big\{ x\in [0,1]^d: \inf_{y\in Q_\ell} \ms{dist} (x,y) \le \sqrt{d}b_n^{-1/d}\big\}. 
$$
Then, the number $D_d$ of such cubes in $S_\ell$ is finite, depending only on  $d$. 
Now, we can offer the following bound: 
$$
U_n^{(1)} \le  \sum_{m=1}^{D_d} \sum_{\ell=1}^{\lfloor D_d^{-1}(b_n-m) \rfloor +1} \hspace{-20pt}\sum_{X \in \Pn} \one \{ X\in Q_{(\ell-1)D_d+m} \}\times \one \big\{  R_k(X, \Pn)   \in \big(  r_n(w_n), \sqrt{d}b_n^{-1/d} \big] \big\}. 
$$
By the homogeneity of $\Pn$ on the torus, we consider only the case $m=D_d$ and obtain from Markov's inequality that,  for every $a>0$, 
\begin{align}
&\frac{1}{b_n} \log \P \bigg(  \sum_{\ell=1}^{\lfloor D_d^{-1}b_n \rfloor } \hspace{-5pt}\sum_{X\in\Pnr_{ Q_{\ell D_d}}} \hspace{-10pt}\one \big\{   R_k(X, \Pn)   \in \big(  r_n(w_n), \sqrt{d}b_n^{-1/d} \big] \big\} \ge \delta b_n\bigg) \label{e:Un1.log}\\
&\le -a\delta + \frac{1}{b_n} \log \E \Big[ e^{a \sum_{\ell=1}^{\lfloor D_d^{-1}b_n \rfloor } \sum_{X\in\Pnr_{Q_{\ell D_d}}}\one \big\{ R_k(X, \Pn)) \in (r_n(w_n), \sqrt{d}b_n^{-1/d}] \big\} } \Big]. \notag 
\end{align}
Here, a key observation is that if $X\in \Pnr_{Q_{\ell D_d}}$ with $R_k(X,\Pn)\le \sqrt{d}b_n^{-1/d}$, then $R_k(X,\Pn) = R_k (X, \Pnr_{S_{\ell D_d}})$, such that   $\big(  \Pnr_{S_{\ell D_d}}, \, \ell=1,\dots,\lfloor D_d^{-1}b_n \rfloor \big)$ are i.i.d.~Poisson point processes. Hence, \eqref{e:Un1.log} can be further bounded by 
\begin{align*}
&-a\delta + D_d^{-1} \log \E\Big[ e^{a\sum_{X\in \Pnr_{Q_1}}\one \big\{ R_k(X,\Pn) \in (r_n(w_n), \sqrt{d}b_n^{-1/d}] \big\} } \Big] \\
&\le -a\delta + D_d^{-1} \log \E\Big[ e^{a\sum_{X\in \Pnr_{Q_1}}\one \big\{ \Pn (B_{r_n(w_n)}(X))\le k  \big\} } \Big]. 
\end{align*}
For the inequality above, we have  applied \eqref{e:K.and.Pn1} and   dropped the condition $R_k(X, \Pn)\le \sqrt{d}b_n^{-1/d}$. Now, it remains to show that for every $a>0$, 
$$
\limsup_{n\to\infty} \E \Big[ e^{a\sum_{X\in\Pnr_{Q_1}} \one \big\{ \Pn(B_{r_n(w_n)}(X)) \le k \big\} } \Big] \le 1. 
$$
The proof is however a simple repetition of the argument for \eqref{e:UI2}, so we  skip it here. 

Returning to \eqref{e:Unj.negligibility}, we next work with  the case $j=2$. To this aim, we exploit an argument similar to that in \cite[Proposition 5.6]{hirsch:owada:2022}.  For $1\le j \le d$, define the  collection of ordered $j$-tuples
$$
\I_j = \big\{ \bell=(\ell_1,\dots,\ell_j): 1\le \ell_1 < \dots < \ell_j \le d \big\}. 
$$
Given $\bell=(\ell_1,\dots,\ell_j)\in \I_j$, define also the collection of hyper-rectangles by 
\begin{align*}
J_n(r) &:= \bigg\{  \Big(  b_n^{-1/d}z  +\big[0,b_n^{-1/d}\big]^{\ell_1-1}\times [-r,r]\times \big[0,b_n^{-1/d}\big]^{\ell_2-\ell_1-1} \\
&\qquad  \times [-r,r] \times \big[0,b_n^{-1/d}\big]^{\ell_3-\ell_2-1} \times \dots \times [-r,r]   \\
&\qquad \times \big[0,b_n^{-1/d}\big]^{\ell_j-\ell_{j-1}-1} \times [-r,r] \times \big[0,b_n^{-1/d}\big]^{d-\ell_j} \Big) \cap [0,1]^d: z\in \bbz_+^d \bigg\}, \ \ \ r>0. 
\end{align*}
By construction, all the rectangles in $J_n(r)$ are contained in $\bigcup_{\ell=1}^{b_n} Q_\ell^\partial (r)$, and the number of rectangles in $J_n(r)$ is  $b_n$; hence, we can enumerate these rectangles  as 
$$
J_n(r)= \big( I_{p,n}^\bell(r), \, p=1,\dots,b_n \big). 
$$
In this setting,  one can bound $U_n^{(2)}$ by 
\begin{align*}
&\sum_{j=1}^d \sum_{\bell \in \I_j} \sum_{p=1}^{b_n} \sum_{X\in\Pn} \one \Big\{ X \in I_{p,n}^\bell \big( r_n(w_n) \big) \Big\}\, \\
&\qquad \qquad \qquad\qquad  \times \one \Big\{ R_k(X,\Pn) > r_n(s_0), \, B_{R_k(X,\Pn)}(X)\subset I_{p,n}^\bell \big( 2r_n(w_n) \big) \Big\}. 
\end{align*}
Owing to this bound, we need to prove that for every $j\in \{ 1,\dots,d \}$, $\bell\in \I_j$, and $\delta>0$,
\begin{align}
&\frac{1}{b_n}\log \P \bigg( \sum_{p=1}^{b_n} \sum_{X\in\Pn}  \one \Big\{ X \in I_{p,n}^\bell \big( r_n(w_n) \big) \Big\} \label{e:Un2.bound}\\
&\qquad  \qquad \times \one \Big\{ R_k(X,\Pn) > r_n(s_0), \, B_{R_k(X,\Pn)}(X)\subset I_{p,n}^\bell \big( 2r_n(w_n) \big) \Big\} \ge \delta b_n\bigg) \to -\infty, \ \ \ n\to\infty. \notag  
\end{align}
In the above, $B_{R_k(X, \Pn)}(X)\subset I_{p,n}^\bell \big( 2r_n(w_n) \big)$ with $X\in I_{p,n}^\bell \big( r_n(w_n) \big)$, implies that $R_k(X,\Pn) = R_k\big( X,\Pnr_{I_{p,n}^\bell (2r_n(w_n))} \big)$. 
Additionally,  $\big( I_{p,n}^\bell (2r_n(w_n))\big)_{p=1}^{b_n}$ are disjoint sets, so $\big( \Pnr_{I_{p,n}^\bell (2r_n(w_n))} \big)_{p=1}^{b_n}$ becomes a sequence of  i.i.d.~Poisson point processes. Hence, by appealing to Markov's inequality as well as \eqref{e:K.and.Pn1}, one can bound  \eqref{e:Un2.bound} by 
\begin{align*}
&-a\delta +\log \E\Big[ e^{a  \sum_{X\in\Pn} \one \big\{ X \in I_{1,n}^\bell( r_n(w_n))\big\} \times \one \big\{ R_k(X,\Pn)>r_n(s_0), \, B_{R_k(X,\Pn)}(X)\subset I_{1,n}^\bell (2r_n(w_n)) \big\}   } \Big] \\
&\le -a\delta +\log \E\Big[ e^{a  \sum_{X\in\Pn} \one \big\{ X \in I_{1,n}^\bell( r_n(w_n)) \big\}\times \one \big\{ \Pn(B_{r_n(s_0)}(X)) \le k\big\}  } \Big]. 
\end{align*}
It is now enough to demonstrate that, for every $a>0$, 
$$
\limsup_{n\to\infty} \E \Big[ e^{a\sum_{X\in\Pn}\one\big\{ X \in I_{1,n}^\bell( r_n(w_n))\big\}  \times \one \big\{ \Pn(  B_{r_n(s_0)}(X))\le k \big\}  } \Big]=1. 
$$
Since the required uniform integrability has already been proven by \eqref{e:UI1}, it suffices to show that as $n\to\infty$,
$$
\E \Big[ \sum_{X\in\Pn}  \one \Big\{ X \in I_{1,n}^\bell \big( r_n(w_n) \big) \Big\}\times \one \Big\{ \Pn \big(  B_{r_n(s_0)}(X)\big)\le k \Big\}\Big]\to 0. 
$$
Proceeding  as before, we obtain that 

\begin{align*}
&\E \Big[ \sum_{X\in\Pn}  \one \Big\{ X \in I_{1,n}^\bell \big( r_n(w_n) \big) \Big\}\times \one \Big\{ \Pn \big(  B_{r_n(s_0)}(X)\big)\le k \Big\}\Big] \\
&= n e^{-(a_n+s_0)} \sum_{i=0}^{k-1} \frac{(a_n+s_0)^i}{i!}\, \P \big( X_1\in I_{1,n}^\bell (r_n(w_n)) \big) \\
&= n e^{-(a_n+s_0)} \sum_{i=0}^{k-1} \frac{(a_n+s_0)^i}{i!} \times C^*r_n(w_n)^j b_n^{-(d-j)/d} \\
&\le C^*\big( b_nr_n(w_n)^d \big)^{j/d} \to 0, \ \ \ n\to\infty. 
\end{align*}
\end{proof}

\begin{proof}[Proof of Corollary \ref{c:LDP.LknB}]
Since the desired  LDP has already been shown for the case of a Poisson input, it is sufficient to demonstrate that for every $\vep_0>0$, 
\begin{equation}  \label{e:diff.Poisson.binomial}
b_n^{-1} \log \P \big( d_{\ms{TV}} (L_{k,n}, L_{k,n}^\ms{B}) \ge \vep_0 b_n \big) \to -\infty, \ \  \text{as } n\to\infty. 
\end{equation}
Our proof is inspired by Corollary 2.3 in \cite{hirsch:owada:2022}. First, define  
{
\begin{equation}  \label{e:diluted.cubes}
G:= \big\{ 3r_n(w_n)z + [0, r_n(s_0)/\sqrt{d}]^d \subset [0,1]^d: z\in \bbz^d  \big\}, 
\end{equation}
}
and consider \emph{finitely many} translates of $G$, denoted $G_1, G_2, \dots, G_M$ for some $M$, such that $[0,1]^d$ can be covered by the union of these translates. In particular, we set $G_1=G$ and denote it specifically as $G=\{ J_1,\dots,J_{b_n'} \}$, where $J_1=[0,r_n(s_0)/\sqrt{d}]^d$ {and $b_n' := \big(  3r_n(w_n) \big)^{-d}$} denotes the number of cubes in $G$.  Since $M$ is a finite constant, \eqref{e:diff.Poisson.binomial}  follows if one can show that  
\begin{align*}
&b_n^{-1}\log \P \Big(  \sup_{A\subset E_0} \Big| \, \sum_{X\in\Pn} g(X, \Pn)\, \one \Big\{ X \in \bigcup_{\ell=1}^{b_n'} J_\ell \Big\} \, \delta_{(X, f(X, \Pn))}(A) \\
&\qquad \qquad \qquad  - \sum_{X\in \B_n} g(X, \B_n)\, \one \Big\{ X \in \bigcup_{\ell=1}^{b_n'} J_\ell \Big\} \, \delta_{(X, f(X, \B_n))}(A)\, \Big| \ge \vep_0b_n \Big) \to -\infty. 
\end{align*}
We say that $J_i$ is \emph{$n$-bad} if one of the following events happens. 
\vspace{7pt}

\noindent $(i)$ There exists $X\in \Pn \cap J_i$ such that $g(X, \Pn)=1$ (equivalently, $\Pn \big( B_{r_n(s_0)}(X) \big) \le k$; see \eqref{e:K.and.Pn1}) and $X\notin \B_n$. \\
$(ii)$ There exists $X\in \B_n \cap J_i$ such that $g(X, \B_n)=1$ and $X\notin \Pn$. \\
$(iii)$ There exist $X\in \Pn \cap \B_n \cap J_i$ and $u\ge s_0$ such that $\min\big\{ \Pn \big( B_{r_n(u)}(X) \big),  \B_n \big( B_{r_n(u)}(X) \big)\big\}\le k$ and 
$\max\big\{\Pn \big( B_{r_n(u)}(X) \big), \B_n \big( B_{r_n(u)}(X) \big)\big\}> k$. 
\vspace{7pt}

\noindent The key observation here is that 
\begin{align*}
&\sup_{A\subset E_0} \Big| \, \sum_{X\in\Pn} g(X, \Pn)\, \one \Big\{ X \in \bigcup_{\ell=1}^{b_n'} J_\ell \Big\} \, \delta_{(X, f(X, \Pn))}(A) \\
&\qquad - \sum_{X\in \B_n} g(X, \B_n)\, \one \Big\{ X \in \bigcup_{\ell=1}^{b_n'} J_\ell \Big\} \, \delta_{(X, f(X, \B_n))}(A)\, \Big| \le (k+1)\sum_{i=1}^{b_n'} \one \{ J_i \text{ is } n\text{-bad} \}. 
\end{align*}
Thus, it is enough to show that for every $\vep_0>0$, 
$$
\frac{1}{b_n} \log \P \Big( \sum_{i=1}^{b_n'} \one \{ J_i \text{ is } n\text{-bad} \} \ge \vep_0b_n\Big) \to -\infty. 
$$

For $\eta\in (0,1]$, let $\Pn^{(\eta)}$ be a homogeneous Poisson point process on $[0,1]^d$ with intensity $n\eta$. We take $\Pn^{(\eta)}$  to be  independent of $\Pn$. Then, $\Pn^{(\eta, \ms{a})}:= \Pn \cup \Pn^{(\eta)}$ represents the \emph{augmented} Poisson point process with intensity $n(1+\eta)$. Moreover, let $\Pn^{(\eta, \ms{t})}$ denote a \emph{thinned} version of $\Pn$ obtained by removing each point of $\Pn$ with probability $\eta$. 
If we denote by $\D_\eta(\Pn)$ a collection of deleted points of $\Pn$, one can write $\Pn^{(\eta, \ms{t})}=\Pn \setminus \D_\eta(\Pn)$. 
Notice that $\Pn^{(\eta, \ms{a})}\stackrel{d}{=}\mathcal P_{n(1+\eta)}$ and $\Pn^{(\eta, \ms{t})}\stackrel{d}{=}\mathcal P_{n(1-\eta)}$. Subsequently, for $\vep>0$ let 
\begin{equation}  \label{e:def.Fnvep}
F_{n,\vep}= \{ \Pn^{(\vep a_n^{-1}, \ms{t})} \subset \B_n \subset \Pn^{(\vep a_n^{-1}, \ms{a})}  \}
\end{equation}
and claim that 
\begin{equation}  \label{e:Fn.vep.negligible}
b_n^{-1}\log \P(F_{n,\vep}^c)\to-\infty,  \ \ \ n\to\infty. 
\end{equation}
For the proof we use  Lemma 1.2 in \cite{penrose:2003} to get that 
\begin{equation}  \label{e:lemma1.2.bound.Fnvep}
\P(F_{n,\vep}^c) \le e^{-n(1+\vep a_n^{-1})H( (1+\vep a_n^{-1})^{-1} )} + e^{-n(1-\vep a_n^{-1})H( (1-\vep a_n^{-1})^{-1} )}, 
\end{equation}
where $H(x)=x\log x +1-x$, $x>0$. Applying the Taylor expansion to $H(\cdot)$, we have 
$$
\limsup_{n\to\infty}\frac{1}{b_n} \log \P (F_{n,\vep}^c) \le -\lim_{n\to\infty} \frac{e^{a_n}}{2a_n^{k-1}} \Big( \frac{\vep a_n^{-1}}{1-\vep a_n^{-1}} \Big)^2=-\infty. 
$$

Suppose now that $J_i$ is an $n$-bad cube  and $F_{n,\vep}$ holds, {such that} one of the events in case $(i)$--$(iii)$ above occurs. Then, there exists $X\in \Pn^{(\vep a_n^{-1}, \ms{a})}\cap J_i$, such that 
\begin{equation}  \label{e:necessary.cond.n-bad}
\Pn^{(\vep a_n^{-1}, \ms{t})} \big( B_{r_n(s_0)}(X) \big) \le k, \ \ \text{ and } \ \ \big(  \Pn^{(\vep a_n^{-1}, \ms{a})}  \setminus \Pn^{(\vep a_n^{-1}, \ms{t})} \big) \big( B_{r_n(w_n)}(X) \big) \ge1.  
\end{equation}
Since we work with the diluted cubes in \eqref{e:diluted.cubes}, it follows from \eqref{e:necessary.cond.n-bad} and the spatial independence of Poisson processes that 
$\sum_{i=1}^{b_n'} \one \{ J_i \text{ is } n\text{-bad} \}$ is a binomial random variable. Below, we shall estimate its success probability $p_{n,\vep}$  as follows:  
\begin{align}
\begin{split}  \label{e:estimate.pnvep}
p_{n,\vep} &= \P \bigg( \bigcup_{X\in \Pn^{(\vep a_n^{-1}, \ms{a})} \cap J_1} \hspace{-10pt}\Big\{ \Pn^{(\vep a_n^{-1}, \ms{t})} \big( B_{r_n(s_0)}(X) \big) \le k, \\
&\qquad \qquad \qquad\qquad \qquad \qquad    \big(  \Pn^{(\vep a_n^{-1}, \ms{a})}  \setminus \Pn^{(\vep a_n^{-1}, \ms{t})} \big) \big( B_{r_n(w_n)}(X) \big) \ge1\Big\} \bigg) \\
&\le \E \bigg[  \sum_{X\in \Pn \cap J_1} \one \Big\{ \big(\Pn \setminus \D_{\vep a_n^{-1}}(\Pn) \big) \big( B_{r_n(s_0)}(X) \big) \le k, \\
&\qquad \qquad \qquad   \qquad \qquad  \qquad \qquad  \big(  \Pn^{(\vep a_n^{-1})}  \cup \D_{\vep a_n^{-1}}(\Pn)\big) \big( B_{r_n(w_n)}(X) \big) \ge1  \Big\}  \bigg] \\
&\quad +  \E \bigg[  \sum_{X\in \Pn^{(\vep a_n^{-1})} \cap J_1} \one \Big\{ \big(\Pn \setminus \D_{\vep a_n^{-1}}(\Pn) \big) \big( B_{r_n(s_0)}(X) \big) \le k \Big\}  \bigg] \\
&=:A_n+B_n. 
\end{split}
\end{align}
By the independence of   $\Pn^{(\vep a_n^{-1})}$ and $\Pn$, 
\begin{align}
\begin{split}  \label{e:calc.Bn}
B_n &= n\vep a_n^{-1} \P \Big( Y\in J_1, \, \big( \Pn\setminus \D_{\vep a_n^{-1}}(\Pn) \big) \big( B_{r_n(s_0)}(Y) \big) \le k\Big) \\
&= n\vep a_n^{-1} \ms{Leb} (J_1)  \P \Big( \mathcal P_{n(1-\vep a_n^{-1})} \big( B_{r_n(s_0)}(Y) \big)\le k \Big) \\
&= n\vep a_n^{-1} \ms{Leb} (J_1) \sum_{i=0}^k \frac{\big( a_n(1-\vep a_n^{-1})(1+s_0a_n^{-1}) \big)^i}{i!}\, e^{-a_n(1-\vep a_n^{-1})(1+s_0a_n^{-1}) }. 
\end{split}
\end{align}
In the above $Y$ is a uniform random variable on $[0,1]^d$, independent of $\Pn$.  Because of 
\begin{equation}  \label{e:s_0.cond}
a_n(1-\vep a_n^{-1})(1+s_0a_n^{-1}) \ge a_n -C^*, 
\end{equation}
we get that 
$$
B_n \le C^* n\vep a_n^{-1} \ms{Leb}(J_1) a_n^k e^{-a_n} = C^* \vep \ms{Leb}(J_1) b_n. 
$$
Applying the Mecke formula for Poisson point processes, 
\begin{align*}
A_n &= n \P \Big(  Y\in J_1, \, \big\{  (\Pn+\delta_Y) \setminus \D_{\vep a_n^{-1}} (\Pn +\delta_Y) \big\} \big( B_{r_n(s_0)}(Y) \big) \le k, \\
&\qquad\qquad \qquad\qquad  \big\{ \Pn^{(\vep a_n^{-1} ) } \cup \D_{\vep a_n^{-1}} (\Pn +\delta_Y) \big\} \big( B_{r_n(w_n)}(Y) \big) \ge 1 \Big). 
\end{align*}
Denote $T(Y)=\{ Y \text{ is deleted by thinning}\}$. Note that if $T(Y)$ holds, then $\D_{\vep a_n^{-1}}(\Pn +\delta_Y)=\D_{\vep a_n^{-1}}(\Pn)+\delta_Y$ and if $T(Y)$ does not hold, we have $\D_{\vep a_n^{-1}}(\Pn +\delta_Y) = \D_{\vep a_n^{-1}}(\Pn)$. Hence,
\begin{align*}
A_n &\le n \P \Big( \Big\{ Y\in J_1, \, \big( \Pn \setminus \D_{\vep a_n^{-1}}(\Pn) \big)\big( B_{r_n(s_0)} (Y)\big) \le k  \Big\} \cap T(Y) \Big) \\
&\quad + n\P \Big( Y\in J_1, \, \big( \Pn \setminus \D_{\vep a_n^{-1}}(\Pn) \big)\big( B_{r_n(s_0)} (Y)\big) \le k-1, \\
&\qquad \qquad \qquad \qquad \big( \Pn^{(\vep a_n^{-1})} \cup \D_{\vep a_n^{-1}} (\Pn)\big)\big( B_{r_n(w_n)} (Y)\big) \ge 1\Big) =: C_n + D_n. 
\end{align*}
Repeating the same calculation as in \eqref{e:calc.Bn} and using \eqref{e:s_0.cond}, 
$$
C_n \le C^* na_n^k e^{-a_n} \P \big( \{  Y\in J_1\}\cap T(Y) \big) = C^* \vep \ms{Leb}(J_1)b_n. 
$$
Since $\Pn \setminus \D_{\vep a_n^{-1}}(\Pn)$ and $\Pn^{(\vep a_n^{-1})}\cup \D_{\vep a_n^{-1}}(\Pn)$ are independent (see, e.g., Corollary 5.9 in \cite{last:penrose:2017}), 
\begin{align*}
D_n &= n \ms{Leb}(J_1) \P \Big( \big( \Pn \setminus \D_{\vep a_n^{-1}}(\Pn) \big) \big( B_{r_n(s_0)}(Y) \big) \le k-1 \Big) \\
&\qquad \qquad  \times  \P\Big(\big( \Pn^{(\vep a_n^{-1})}\cup \D_{\vep a_n^{-1}}(\Pn) \big) \big( B_{r_n(w_n)}(Y) \big) \ge 1 \Big) \\
&= n\ms{Leb}(J_1) \P \Big( \mathcal P_{n(1-\vep a_n^{-1})} \big( B_{r_n(s_0)}(Y) \big) \le k-1 \Big)  \P\Big(\mathcal P_{2n\vep a_n^{-1}} \big( B_{r_n(w_n)}(Y) \big) \ge 1 \Big) 
\end{align*}
By \eqref{e:s_0.cond} and Markov's inequality, 
\begin{align*}
D_n &\le C^* n \ms{Leb}(J_1) \, a_n^{k-1} e^{-a_n} \E \Big[ \mathcal P_{2n\vep a_n^{-1}} \big( B_{r_n(w_n)}(Y) \big) \Big] \\
&= C^* n \ms{Leb}(J_1) \, a_n^{k-1} e^{-a_n}  2 \vep (1+w_na_n^{-1}) \le C^* \vep \ms{Leb}(J_1) b_n. 
\end{align*}
Combining all these calculations concludes that $p_{n, \vep }\le C^* \vep \ms{Leb}(J_1) b_n$. If one takes sufficiently small $\vep\in (0,\vep_0)$, we have that $b_n' p_{n,\vep} \le C^* b_n' \vep \ms{Leb}(J_1)b_n = C^* \vep b_n \le \vep_0 b_n$ for large $n$ enough. Therefore, one can exploit the binomial concentration inequality (see, e.g., Lemma 1.1 in \cite{penrose:2003}) to obtain that 
\begin{align*}
\limsup_{n\to\infty} b_n^{-1} \log \P \Big( \text{Bin} (b_n', C^* \vep \ms{Leb}(J_1)b_n) \ge \vep_0 b_n \Big) &\le -\frac{\vep_0}{2}\, \lim_{n\to\infty} \log \Big\{ \frac{\vep_0}{C^* b_n' \vep \ms{Leb}(J_1)} \Big\} \\
&= {-\frac{\vep_0}{2}\, \log \Big\{ \frac{(3\sqrt{d})^d\vep_0}{C^* \vep} \Big\}.}
\end{align*}
The last term goes to $-\infty$ as $\vep \to 0$. Combining this result with \eqref{e:Fn.vep.negligible} concludes the proof.
\end{proof}

\begin{proof}[Proof of Corollary \ref{c:LDP.application}]
Define the map $S: M_+(E_0) \to [0,\infty)$ by $S(\rho)=\rho(E_0)$. Since $S$ is continuous in the weak topology {and $H_k(\cdot\,|\, \ms{Leb}\otimes \tau_k)=\Lambda_k^*$ is a good rate function}, the contraction principle (see, e.g., \cite[Theorem 4.2.1]{dembo:zeitouni:1998}) is applied  to the LDP in Theorem \ref{t:LDP.L.kn}. In conclusion,  
$$
S\Big(\frac{L_{k,n}}{b_n}\Big) = \frac{1}{b_n}\sum_{X\in \Pn}g(X,\Pn) = \frac{T_{k,n}}{b_n}, \  \ n\ge1, 
$$
shows an LDP with rate $b_n$ and rate function 
\begin{equation}  \label{e:rate.func.contraction.principle}
\inf_{\nu\in M_+(E_0), \, \nu(E_0)=x} H_k(\nu | \ms{Leb}\otimes \tau_k), \ \ \ x\in \R, 
\end{equation}
where $H_k$ is the relative entropy defined at \eqref{e:def.relative.entropy}. The rest of the argument must be devoted to verifying that \eqref{e:rate.func.contraction.principle} coincides with $I_k(x)$ for every $x\in \R$; this is however an immediate result as an analogue of  Equ.~(5.35) of \cite{hirsch:owada:2022}.  

The LDP for $(T_{k,n}^\ms{B}/b_n)_{n\ge1}$ is obtained by applying the  contraction principle to Corollary \ref{c:LDP.LknB}. 
\end{proof}
\medskip

\subsection{Proofs of Theorem \ref{t:M0.L.kn}, Corollary \ref{c:M0.LknB}, and Corollary \ref{c:M0.application}}   \label{sec:proof.M0}

\begin{proof}[Proof of Theorem \ref{t:M0.L.kn}]
Let $C_K^+(E)$ denote a collection of continuous and non-negative functions on $E$ with compact support.  
Given two such functions $U_\ell \in C_K^+(E)$, $\ell=1,2$,  together with $\vep_1, \vep_2 > 0$, we define $F_{U_1, U_2, \vep_1, \vep_2}: M_p(E) \to [0,1]$ by 
\begin{equation}  \label{e:def.F}
F_{U_1, U_2, \vep_1, \vep_2} (\eta) = \Big( 1-e^{-(\eta(U_1) - \vep_1)_+} \Big)\Big( 1-e^{-(\eta(U_2) - \vep_2)_+} \Big), 
\end{equation}
where $\eta(U_\ell) = \int_E U_\ell(x,u) \eta (\dif x, \dif u)$, {and $(a)_+=a$ if $a\ge0$ and $0$ otherwise}. 
Notice that $F_{U_1, U_2, \vep_1, \vep_2}\in \mathcal C_0$. 
In what follows, we fix $U_1, U_2$ and $\vep_1, \vep_2$, and simply write $F=F_{U_1, U_2, \vep_1, \vep_2}$. Define 

$\xi_{k,n} (\cdot):= b_n^{-1}\P ( L_{k,n} \in \cdot )$. 
Then, according to Theorem A.2 in \cite{hult:samorodnitsky:2010},  \eqref{e:M0.convergence} follows if one can show that 
$$
\xi_{k,n} (F) \to \xi_k (F), \ \ \text{as } n\to\infty. 
$$
First, note that
$$
\xi_{k,n}(F) = \int_{M_p(E)} F(\eta) \xi_{k,n}(\dif \eta) = b_n^{-1} \E \big[ F(L_{k,n}) \big]. 
$$
Let $\zeta_{k,n}$ denote a Poisson point process on $E$ with mean measure 
\begin{equation}  \label{e:mean.meas.zeta.kn.original}
\frac{b_n}{(k-1)!}\, e^{-u} \dif x \dif u, \ \ x\in [0,1]^d, \ u\in \R. 
\end{equation}
In this setting, our proof breaks down into two parts: 
\begin{align}
&b_n^{-1} \big|\,  \E\big[ F(L_{k,n}) \big] - \E \big[ F(\zeta_{k,n}) \big]\, \big| \to 0, \ \ \ n\to\infty, \label{e:1st.M0}  \\
&b_n^{-1} \E \big[ F(\zeta_{k,n}) \big] \to \xi_k(F),  \ \ \ n\to\infty.   \label{e:2nd.M0}
\end{align}
\textit{Proof of \eqref{e:1st.M0}}: Since $U_\ell$  has compact support on $E$, there exists $s_0\in \R$, so that 
$$
 \text{supp} (U_1)\bigcup \text{supp} (U_2)\subset [0,1]^d \times (s_0,\infty], 
$$
where $\text{supp}(U_\ell)$ represents the support of $U_\ell$. 
Hence, we may assume, without loss of generality, that $L_{k,n}$ and $\zeta_{k,n}$ are both random elements of the \emph{restricted} state space $M_p\big( [0,1]^d \times (s_0, \infty] \big)$. Equivalently, one can reformulate $L_{k,n}$ by
\begin{equation}  \label{e:Lkn.reformulate.M0}
L_{k,n} = \begin{cases}
\sum_{X\in\Pn} g(X,\Pn)\, \delta_{(X, f(X,\Pn))} & \text{ if } |\Pn| >k,\\
\emptyset & \text{ if } |\Pn| \le k, 
\end{cases}
\end{equation}
in the same way as \eqref{e:def.Lkn2}. Similarly,  $\zeta_{k,n}$ can be defined as the Poisson point process whose mean measure is given by the restricted version of \eqref{e:mean.meas.zeta.kn.original}; that is, 
$$
(\Leb\otimes \tau_{k,n}) (\dif x, \dif u)  :=\frac{b_n}{(k-1)!}\, e^{-u} \one \{ u\ge s_0 \}\dif x \dif u,  \ \ x\in [0,1]^d, \ u \in \R. 
$$

Next, it is not hard to prove that $F$  in \eqref{e:def.F} is a $1$-Lipschitz function with respect to the total variation distance on the space of point measures. Namely,  for $\eta_1, \eta_2 \in M_p(E)$, 
$$
\big| F(\eta_1)-F(\eta_2) \big| \le 2 d_{\ms{TV}}(\eta_1,\eta_2). 
$$
Thus, by \eqref{e:def.KR.dist}, 
$$
\big|\,  \E\big[ F(L_{k,n}) \big] - \E \big[ F(\zeta_{k,n}) \big]\, \big|  \le  d_{\ms{KR}} \big( \mathcal L(L_{k,n}), \mathcal L(\zeta_{k,n}) \big). 
$$

\begin{proposition}    \label{p:KR.conv2}
We have, as $n\to\infty$, 
$$
b_n^{-1} d_{\ms{KR}} \big( \mathcal L(L_{k,n}), \mathcal L(\zeta_{k,n}) \big)\to 0. 
$$
\end{proposition}
\begin{proof}[Proof of Proposition \ref{p:KR.conv2}]
The proof is analogous to that of Proposition \ref{p:KR.conv1}. Precisely, we first need to show that 
$$
b_n^{-1} d_{\ms{TV}} \big( \E [L_{k,n}(\cdot)], \Leb \otimes \tau_{k,n} \big) \to 0, \ \ \ n\to\infty, 
$$
and verify also that  $b_n^{-1}E_i \to 0$, $n\to\infty$, for $i=1,2,3$, where $E_i$'s are defined analogously  to \eqref{e:E1}, \eqref{e:E2}, and \eqref{e:E3}. More concretely, they are respectively defined as 
$$
E_1 := 2n\int_{[0,1]^d} \E \big[  g(x,\Pn+\delta_x) \, \one \big\{  \mathcal S (x,\Pn+\delta_x) \not\subset S_x \big\} \big]\dif x, 
$$
$$
E_2 := 2n^2 \int_{[0,1]^d}\int_{[0,1]^d}  \hspace{-.2cm}\one \{ S_x \cap S_z \neq \emptyset \} \,\E \big[  g(x,\Pn+\delta_x) \big]  \E \big[  g(z,\Pn+\delta_z) \big] \dif x \dif z, 
$$
and 
$$
	E_3 := 2n^2 \int_{[0,1]^d}\int_{[0,1]^d} \hspace{-.2cm}\one \{ S_x \cap S_z \neq \emptyset \} \,\E \big[  g(x,\Pn+\delta_x+\delta_z) \,  g(z,\Pn+\delta_x +\delta_z) \big] \dif x \dif z, 
$$
for which $\mathcal S(x,\omega)=B_{R_k(x,\omega)}(x)$ for $x\in [0,1]^d$ and  $\omega\in M_p\big( [0,1]^d \big)$, and $S_x = B_{r_n(w_n)}(x)$ for some sequence $w_n\to\infty$ with $w_n=o(a_n)$, $n\to\infty$. 

First, for  $B\subset [0,1]^d$ and $u>s_0$, by the Mecke formula for Poisson point processes and \eqref{e:K.and.Pn1},
$$
\E\big[ L_{k,n}(B\times (u,\infty)) \big] = n\P\Big(Y\in B, \, (\Pn +\delta_Y ) \big( B_{r_n(u)}(Y) \big) \le k  \Big), 
$$
where $Y$ is a uniform random variable on $[0,1]^d$, independent of $\Pn$. By the conditioning on $Y$, 
\begin{align*}
\E\big[ L_{k,n}(B\times (u,\infty)) \big]  &= n\, \Leb(B)\sum_{i=0}^{k-1} e^{-(a_n+u)} \frac{(a_n+u)^i}{i!}. 
\end{align*}
This means that $\E\big[ L_{k,n}(\cdot) \big]$ has the density 
$$
n \, \frac{e^{-(a_n+u)}(a_n+u)^{k-1}}{(k-1)!}, \ \ x \in [0,1]^d, \ u >s_0, 
$$
and hence, it follows from the dominated convergence theorem that 
\begin{align*}
&b_n^{-1} d_{\ms{TV}} \big( \E [L_{k,n}(\cdot)], \, \Leb\otimes \tau_{k,n} \big)  \\
&\le b_n^{-1} \int_{[0,1]^d \times (s_0,\infty)} \Big| \, ne^{-(a_n+u)}\frac{(a_n+u)^{k-1}}{(k-1)!} - b_n \frac{e^{-u}}{(k-1)!}  \,  \Big|\dif x \dif u \\
&= \frac{1}{(k-1)!}\, \int_{s_0}^\infty \Big|   \, \Big( 1+\frac{u}{a_n} \Big)^{k-1}-1  \,  \Big| e^{-u}\dif u \to0,  \ \ \ n\to\infty. 
\end{align*}
Subsequently, 
\begin{align*}
b_n^{-1}E_1 &\le 2nb_n^{-1} \int_{[0,1]^d} \P \Big( \Pn \big( B_{r_n(w_n)}(x) \big) \le k-1 \Big)\dif x\\
&=2nb_n^{-1} \sum_{i=0}^{k-1}e^{-(a_n+w_n)} \frac{(a_n+w_n)^i}{i!} \le C^* e^{-w_n} \to 0, \ \ \ n\to\infty, 
\end{align*}
while we also have 
$$
b_n^{-1} E_2 \le 2n^2 b_n^{-1} \Big\{ \sum_{i=0}^{k-1} e^{-(a_n+s_0)} \frac{(a_n+s_0)^i}{i!} \Big\}^2 \le C^*b_n\to 0, \ \ \ n\to\infty. 
$$
Similarly to \eqref{e:E3.split}, $E_3$ can be split into two additional terms: 
\begin{align*}
E_3 &\le 2n^2 \int_{[0,1]^d}\int_{[0,1]^d}\one \big\{ \| x-z\| \le r_n(s_0) \big\}  \\
&\qquad \qquad \qquad \times \P \Big( (\Pn+\delta_z) \big( B_{r_n(s_0)}(x) \big) \le k -1, \,   (\Pn+\delta_x) \big( B_{r_n(s_0)}(z) \big) \le k -1\Big) \dif x \dif z \\
&\quad + 2n^2 \int_{[0,1]^d}\int_{[0,1]^d}\one \big\{r_n(s_0) < \| x-z\| \le 2r_n(w_n) \big\}  \\
&\qquad \qquad \qquad \times \P \Big( (\Pn+\delta_z) \big( B_{r_n(s_0)}(x) \big) \le k -1, \,   (\Pn +\delta_x) \big( B_{r_n(s_0)}(z) \big) \le k -1\Big) \dif x \dif z \\
&=: E_{3,1}+E_{3,2}. 
\end{align*}
Although we shall skip  detailed discussions,  one can still demonstrate  that $b_n^{-1}E_{3,1}\to0$ and $b_n^{-1}E_{3,2}\to0$, by the arguments nearly identical to those for Proposition \ref{p:KR.conv1}. 
\end{proof}

Now, the proof of Proposition \ref{p:KR.conv2} has been completed, which in turn concludes  \eqref{e:1st.M0}. Our next goal is to prove \eqref{e:2nd.M0}.

\textit{Proof of \eqref{e:2nd.M0}}: Note that $\zeta_{k,n}$ can be written as 
$$
\zeta_{k,n} = \sum_{i=1}^{N_n} \delta_{(T_i, Z_i)},
$$ 
where $(T_i, Z_i)$ are i.i.d.~random variables on $E$ with density given by $e^{-(u-s_0)} \one \{ u\ge s_0 \} \dif x\dif u$, and $N_n$ is Poisson distributed with mean $b_ne^{-s_0}/(k-1)!$. Furthermore, $(T_i, Z_i)$ and $N_n$ are taken to be independent.  Substituting this representation, 
\begin{align*}
b_n^{-1}\E \big[ F(\zeta_{k,n}) \big] &= b_n^{-1} \E \Big[  \prod_{\ell=1}^2 \Big( 1-e^{-\big( \sum_{i=1}^{N_n} U_\ell (T_i, Z_i) -\vep_\ell\big)_+} \Big) \Big] \\
&=b_n^{-1} \E \Big[  \prod_{\ell=1}^2 \Big( 1-e^{-\big(  U_\ell (T_1, Z_1) -\vep_\ell\big)_+} \Big) \, \one \{ N_n=1 \}\Big] \\
&\qquad \qquad + b_n^{-1} \E \Big[  \prod_{\ell=1}^2 \Big( 1-e^{-\big( \sum_{i=1}^{N_n} U_\ell (T_i, Z_i) -\vep_\ell\big)_+} \Big) \one \{ N_n\ge2 \}\Big]\\
&=: A_n+B_n. 
\end{align*}
Of the last two terms, one can immediately show that 
$$
B_n \le b_n^{-1} \P(N_n\ge 2) \le \Big( \frac{e^{-s_0}}{(k-1)!} \Big)^2 b_n\to 0, \ \ \text{as } n\to\infty. 
$$
By the independence of $(T_1, Z_1)$ and $N_n$, we have as $n\to\infty$, 
\begin{align*}
A_n &= b_n^{-1} \E \Big[  \prod_{\ell=1}^2 \Big( 1-e^{-\big(  U_\ell (T_1, Z_1) -\vep_\ell\big)_+} \Big) \Big] \P(N_n=1) \\
&= \frac{e^{-s_0}}{(k-1)!}\, e^{-e^{-s_0}b_n/(k-1)!} \int_E \prod_{\ell=1}^2 \Big( 1-e^{-\big( U_\ell(x,u) -\vep_\ell\big)_+} \Big) e^{-(u-s_0)} \one \{ u\ge s_0 \} \dif x \dif u \\
&\to \frac{1}{(k-1)!} \int_E \prod_{\ell=1}^2\Big( 1-e^{-\big( U_\ell(x,u) -\vep_\ell\big)_+} \Big) e^{-u} \dif x \dif u = \xi_k(F). 
\end{align*}
We thus conclude that $A_n+B_n \to \xi_k(F)$, $n\to\infty$, as required. 
\end{proof}

\begin{proof}[Proof of Corollary \ref{c:M0.LknB}]

Because of \eqref{e:1st.M0} and \eqref{e:2nd.M0}, it is sufficient to show that 
$$
b_n^{-1} \E \Big[ \, \big| F(L_{k,n})-F(L_{k,n}^\ms{B})   \big|\, \Big] \to 0, \ \ \ n\to\infty, 
$$
where $F$ is defined at \eqref{e:def.F}. 
Under the map $F$, one can represent $L_{k,n}$ as in \eqref{e:Lkn.reformulate.M0}. Clearly, $L_{k,n}^\ms{B}$ has the same representation as an element of $M_p\big( [0,1]^d\times (s_0,\infty] \big)$. 
Since $F$ is bounded,
$$
b_n^{-1} \E \Big[ \, \big| F(L_{k,n})-F(L_{k,n}^\ms{B})   \big|\, \Big] \le 2b_n^{-1} \P( L_{k,n}\neq L_{k,n}^\ms{B}). 
$$
We now claim that $b_n^{-1}\P( L_{k,n}\neq L_{k,n}^\ms{B} )\to 0$ as $n\to\infty$. The proof is analogous to that of \eqref{e:diff.Poisson.binomial} by  borrowing  the idea of $n$-bad cubes. Specifically, we say that $[0,1]^d$ is \emph{$n$-bad} if one of the following events occurs. 
\vspace{7pt}

\noindent $(i)$ There exists $X\in \Pn$ such that $g(X, \Pn)=1$ and $X\notin \B_n$. \\
$(ii)$ There exists $X\in \B_n$ such that $g(X, \B_n)=1$ and $X\notin \Pn$. \\
$(iii)$ There exist $X\in \Pn \cap \B_n$ and $u\ge s_0$ such that $\min\big\{ \Pn \big( B_{r_n(u)}(X) \big),  \B_n \big( B_{r_n(u)}(X) \big)\big\}\le k$ and 
$\max\big\{\Pn \big( B_{r_n(u)}(X) \big), \B_n \big( B_{r_n(u)}(X) \big)\big\}> k$. 
\vspace{7pt}

\noindent The key observation is that  $[0,1]^d$ becomes $n$-bad whenever $L_{k,n}\neq L_{k,n}^\ms{B}$. Using this fact, we now need to show that 

\begin{equation}  \label{e:n.bad.cube.M0}
b_n^{-1} \P\big( [0,1]^d \text{ is } n\text{-bad} \big) \to 0, \ \ \ n\to\infty. 
\end{equation}
The first step for the proof of  \eqref{e:n.bad.cube.M0} is to demonstrate that 
\begin{equation}  \label{e:Fnvep.M0}
b_n^{-1}\P(F_{n,\vep}^c) \to 0, \ \ \ n\to\infty, 
\end{equation}
where $F_{n,\vep}$ is given in \eqref{e:def.Fnvep}. 
By virtue of the bound in \eqref{e:lemma1.2.bound.Fnvep}, together with an application of the Taylor expansion to $H(\cdot)$, and the assumption $a_n=o(n^{1/3})$, one can get \eqref{e:Fnvep.M0} as desired. Observe also that if $[0,1]^d$ is $n$-bad under $F_{n,\vep}$, then there exists $X\in \Pn^{(\vep a_n^{-1}, \ms{a})}$ such that \eqref{e:necessary.cond.n-bad} holds. Hence, by \eqref{e:Fnvep.M0} and Markov's inequality, 
\begin{align*}
&b_n^{-1} \P\big( [0,1]^d \text{ is } n\text{-bad}\big) \le b_n^{-1} \P\big( \big\{[0,1]^d \text{ is } n\text{-bad}\big\}\cap F_{n,\vep}\big) + o(1) \\
&\le b_n^{-1} \P \Big(  \bigcup_{X\in \Pn^{(\vep a_n^{-1}, \ms{a})} } \hspace{-10pt}\Big\{ \Pn^{(\vep a_n^{-1}, \ms{t})} \big( B_{r_n(s_0)}(X) \big) \le k, \,  \\
&\qquad \qquad \qquad \qquad\qquad \qquad\big(  \Pn^{(\vep a_n^{-1}, \ms{a})}  \setminus \Pn^{(\vep a_n^{-1}, \ms{t})} \big) \big( B_{r_n(w_n)}(X) \big) \ge1\Big\} \Big) +o(1)\\
&\le b_n^{-1}\bigg\{ \E \Big[  \sum_{X\in \Pn} \one \Big\{ \big(\Pn \setminus \D_{\vep a_n^{-1}}(\Pn) \big) \big( B_{r_n(s_0)}(X) \big) \le k, \\
&\qquad \qquad \qquad   \qquad \qquad  \qquad \qquad  \big(  \Pn^{(\vep a_n^{-1})}  \cup \D_{\vep a_n^{-1}}(\Pn)\big) \big( B_{r_n(w_n)}(X) \big) \ge1  \Big\}  \Big] \\
&\quad +  \E \Big[  \sum_{X\in \Pn^{(\vep a_n^{-1})}} \one \Big\{ \big(\Pn \setminus \D_{\vep a_n^{-1}}(\Pn) \big) \big( B_{r_n(s_0)}(X) \big) \le k \Big\}  \Big]\bigg\} +o(1)\\
&=:b_n^{-1}(A_n'+B_n') +o(1). 
\end{align*}
where $\vep\in (0,1)$ is an arbitrary constant. 
Repeating the calculations very similar to those bounding $A_n, B_n$ in  \eqref{e:estimate.pnvep}, one can see that $A_n'+B_n' \le C^* \vep b_n$. Thus, $\limsup_{n\to\infty}b_n^{-1} \P\big( [0,1]^d \text{ is } n\text{-bad}\big) \le C^*\vep$, and letting $\vep \to0$ completes the proof of Corollary \ref{c:M0.LknB}. 
\end{proof}

\begin{proof}[Proof of Corollary \ref{c:M0.application}]
We prove only the first statement. By a straightforward modification of Theorem \ref{t:M0.L.kn} by restricting the state space from $E$ to $E_0=[0,1]^d\times (s_0,\infty]$, we have, as $n\to\infty$, 
$$
\xi_{k,n}(\cdot):=b_n^{-1}\P(L_{k,n}\in \cdot) \to \xi_k \ \ \text{in } \M_0. 
$$
Due to the change of the state space, $L_{k,n}$ is now formulated as in \eqref{e:Lkn.reformulate.M0}, while the limit $\xi_k$ is taken to be 
$$
\xi_k(\cdot)=\frac{1}{(k-1)!}\int_{E_0} \one \{ \delta_{(x,u)}\in \cdot \} e^{-u} \dif x \dif u. 
$$
Now, we define a map $V:M_p(E_0)\to \bbn:=\{ 0,1,2,\dots \}$ by $V(\rho)=\rho(E_0)$. Here, $\bbn$ is equipped with the discrete topology. 
Since  $V$ is continuous in the weak topology, 
it follows from  \cite[Theorem 2.5]{hult:lindskog:2006a} that 
\begin{equation}  \label{e:another.M0.conv}
\xi_{k,n}\circ V^{-1} \to \xi_k \circ V^{-1}, \ \ \text{ in } \M_0,  \ \ \ n\to\infty.  
\end{equation}

Note that $\one_{[1,\infty)}(x)$ is continuous and bounded on $\bbn$ (in terms of the discrete topology), vanishing in the neighborhood of $0$ (i.e., the origin of $\bbn$). Thus, the $\M_0$-convergence in \eqref{e:another.M0.conv} implies that 
$$
b_n^{-1}\P(T_{k,n}\ge1) = \int_{\bbn} \one_{[1,\infty)}(x) \xi_{k,n}\circ V^{-1}(\dif x) \to \int_{\bbn} \one_{[1,\infty)}(x) \xi_{k}\circ V^{-1}(\dif x) = \alpha_k, 
$$
as desired.
\end{proof}

\noindent \textbf{Acknowledgment}: The third author is thankful for fruitful discussions with Yogeshwaran D.~and Z.~Wei, which has   helped him to deduce the desired $\M_0$-convergence in Theorem \ref{t:M0.L.kn} by means of \cite[Theorem 6.4]{bobrowski:schulte:yogeshwaran:2021}. 

\bibliography{LDP-dense}
\end{document}